\documentclass[11pt,reqno]{amsart}

\usepackage[T1]{fontenc}
\usepackage[utf8]{inputenc}
\usepackage{lmodern,microtype}
\usepackage{comment}
\usepackage{mathtools,amssymb,mathrsfs}
\usepackage[margin=1.05in]{geometry}
\usepackage{booktabs,array,enumitem}
\usepackage[hypertexnames=false,
    backref=page,
    pdftex,
    pdfpagemode=UseNone,
    breaklinks=true,
    extension=pdf,
    colorlinks=true,
    linkcolor=blue,
    citecolor=blue,
    urlcolor=blue,
]{hyperref}

\usepackage[nameinlink,noabbrev]{cleveref}
\usepackage{xcolor}

\numberwithin{equation}{section}
\allowdisplaybreaks[2]
\setlist[enumerate]{label=\textup{(\roman*)},leftmargin=*,itemsep=3pt,topsep=5pt}

\newtheorem{introthm}{Theorem}

\newtheorem{theorem}{Theorem}[section]
\newtheorem{proposition}[theorem]{Proposition}
\newtheorem{lemma}[theorem]{Lemma}
\newtheorem{corollary}[theorem]{Corollary}
\theoremstyle{definition}
\newtheorem{definition}[theorem]{Definition}

\theoremstyle{remark}
\newtheorem{remark}[theorem]{Remark}

\newcommand{\ii}{\sqrt{-1}}
\newcommand{\R}{\mathbb R}
\newcommand{\C}{\mathbb C}
\newcommand{\Z}{\mathbb Z}
\newcommand{\CP}{\mathbb{CP}}
\newcommand{\dd}{\mathrm d}
\newcommand{\dv}{\mathrm dV_g}
\newcommand{\Ch}{\mathrm{Ch}}
\newcommand{\scal}{\operatorname{scal}}
\newcommand{\Ric}{\operatorname{Ric}}
\newcommand{\Vol}{\operatorname{Vol}}
\newcommand{\tr}{\operatorname{tr}}

\newcommand{\tf}{\mathring T}
\newcommand{\ty}{t_{\mathrm Y}}
\newcommand{\Tmass}{\mathcal T}
\newcommand{\degree}{\mathcal D}
\newcommand{\energy}{\mathcal Q}
\newcommand{\A}{\mathcal A}
\newcommand{\eps}{\varepsilon}
\newcommand{\norm}[1]{\lVert #1\rVert}
\newcommand{\abs}[1]{\lvert #1\rvert}

\newcommand{\dstar}{\dd^{*}}

\title[deformed Gauduchon--Yamabe problems]
{deformed Gauduchon--Yamabe problems on Compact Hermitian Manifolds}
\author{Oluwagbenga Joshua Windare}
\address{Dipartimento di Matematica e Informatica ``Ulisse Dini'', Universit\`a degli Studi di Firenze, Viale Morgagni 67/a, 50134 Firenze, Italy}
\email{oluwagbengajoshua.windare@gmail.com}
\thanks{The author was partially supported by GNSAGA of INdAM}
\begin{document}
\raggedbottom
\begin{abstract}
We study a two-parameter deformation of scalar curvature associated with the Gauduchon connections on compact Hermitian manifolds. A central identity expresses its dependence on the connection parameter through the squared norm of the trace-free part of the torsion of the Chern connection. We establish existence criteria for metrics of constant deformed scalar curvature and analyze the global dependence of the deformed Yamabe invariant on the connection parameter, including when the critical threshold is reached. We also study a prescribed-curvature problem in positive conformal classes. Applications include attainment of the critical threshold for an explicit Hopf example in every complex dimension at least two and on non-locally conformally Kähler Iwasawa products.
\end{abstract}
\maketitle
%\tableofcontents
%\clearpage

\section{Introduction}

The classical Yamabe problem asks whether a compact Riemannian conformal class contains a metric of constant scalar curvature.  Its resolution rests on the conformal Laplacian, the critical Sobolev exponent, the sharp Euclidean Sobolev inequality, and the concentration level represented by the round sphere.  The work of Yamabe, Trudinger, Aubin, and Schoen established the existence theory and made the Yamabe functional a basic model for nonlinear geometric analysis; see \cite{Yamabe1960,Trudinger1968,Aubin1976Yamabe,Schoen1984,Aubin1998,LeeParker1987}.

On a complex manifold equipped with a Hermitian metric, there is no single distinguished metric connection. Gauduchon introduced in \cite{Gauduchon1997} a one-parameter family of canonical Hermitian connections. Gauduchon's affine line of canonical Hermitian connections contains the Chern, the Lichnerowicz, and Bismut connections. Their scalar curvatures and torsions record different aspects of non-K\"ahler geometry. The Chern connection is natural from the holomorphic point of view. The Chern--Yamabe problem of Angella--Calamai--Spotti \cite{ACS2017} showed that even the most direct Hermitian analogue of the Yamabe problem differs sharply from the Riemannian one; the sign of the Chern--Gauduchon degree governs the existence theory. Existence was established in the non-positive Chern--Gauduchon degree
case, while the positive degree case was shown to be considerably more delicate.  Related developments include \cite{LejmiMaalaoui2018,CalamaiZou2020,Ho2021,HoShin2021,Fusi2022,AP2022,Yu2025}.  For the wider Gauduchon family, see Gauduchon's foundational work \cite{Gauduchon1977,Gauduchon1984,Gauduchon1997}, Barbaro's Bismut--Yamabe theory \cite{Barbaro2023}, and the prescribed Gauduchon scalar-curvature results of Li--Zhou--Zhou \cite{LiZhouZhou2024}.

A different route to a Yamabe-type Hermitian equation was developed in \cite{APSSW2026}.  There, the Chern scalar curvature is modified by carefully chosen Lee-form and torsion terms so that its conformal transformation has the same critical exponent as the ordinary Yamabe problem. It can therefore be studied by the variational method originating in the work of Aubin \cite{Aubin1998, Aubin1976Yamabe} and developed further by Hebey--Vaugon \cite{HebeyVaugon2001} and others.

Two questions then arise. How does the deformation in \cite{APSSW2026} depend on the choice of canonical Hermitian connection? And what happens at the critical threshold, where strict-inequality existence arguments cease to apply? The first question is geometric: The canonical family allows us to ask which features of the construction depend on the choice of connection, and which part of Hermitian torsion is detected by changing the connection. The second is variational: it asks when equality is unattained, when it is attained, and which geometric parameters separate these possibilities. The equality case was left open beyond $\CP^1$ in \cite[Remark 1.2]{APSSW2026}.

We address these questions together. Let $(M^{2n},J,\omega)$ be a compact connected complex manifold with $n\geq2$. The Gauduchon connection is denoted by $\nabla^r$, $r\in\R$, with $r=1,0,-1$ corresponding respectively to the Chern, Lichnerowicz, and Bismut connections. Set
\[
 A_r=1+(n-1)r,\qquad K_r=3r^2-2r+1.
\]
If $\theta$ is the Lee form, $T^r$ the torsion, and $\scal^r$ the scalar curvature of this connection, we consider
\begin{equation}\label{def}
 \mu^{r,t}=\scal^r+\frac{t-A_r}{n-1}\dstar\theta
             +\frac{t-2A_r}{2K_r}\abs{T^r}^2.
\end{equation}
The scalar curvature convention is twice the complex trace, so that the Chern scalar curvature agrees with the Riemannian scalar curvature on a K\"ahler manifold. We use $t>0$ in all variational assertions. Although \eqref{def} depends on both \(r\) and \(t\), the \(r\)-dependence simplifies to
\begin{equation}\label{eq:intro-reduction}
 \mu^{r,t}=\mu^{1,t}-\frac{n-1}{2}(r-1)\abs{\tf}^2,
\end{equation}
where $\tf$ is the trace-free Chern torsion and $\mu^{1,t}$ is the Chern slice studied in \cite{APSSW2026}. $\tf$ is conformally invariant as a tensor and $\abs{\tf}^2$ has conformal weight $-1$. In complex dimension at least three, its vanishing is equivalent to local conformal K\"ahlerness. In complex dimension two, it vanishes identically for algebraic reasons; this does not assert that every Hermitian surface is locally conformally K\"ahler (LCK).

Let $[\omega]$ denote the conformal class of $\omega$ and set
\[
 N:=\frac{2n}{n-1},
 \qquad
 \Lambda_n:=2n\,\operatorname{Vol}(S^{2n})^{1/n},
\]
The $(r,t)$ deformed Yamabe invariant is defined as
\begin{equation}\label{eq:intro-lambda}
 \lambda_{r,t}([\omega])
 :=\inf_{u\in H^1(M)\setminus\{0\}}
 \frac{\displaystyle\int_M\left(\frac{2t}{n-1}|\nabla u|^2+\mu^{r,t}u^2\right)\frac{\omega^n}{n!}}
 {\displaystyle\left(\int_M|u|^N\frac{\omega^n}{n!}\right)^{2/N}}.
\end{equation}
This is a conformal invariant. Its universal upper bound is $t\Lambda_n$, and strict inequality yields a smooth positive minimizer and therefore a conformal metric of constant $\mu^{r,t}$-curvature (Theorem~ \ref{thm:sharp-threshold}). The first observation is that if $n\ge3$ and $\omega$ is locally conformally K\"ahler, then
\[
 \mu^{r,t}=\mu^{1,t},
 \qquad
 \lambda_{r,t}([\omega])=\lambda_{1,t}([\omega])
\]
for every $r\in\mathbb R$ and every $t>0$.

 Let $\eta$ be the unit-volume Gauduchon representative of $[\omega]$, and define the deformed Gauduchon degree
\[
 \mathcal D_{r,t}([\omega])
 :=\int_M\mu^{r,t}(\eta)\,\mathrm dV_\eta.
\] The main results can then be organized as follows.

\begin{introthm}\label{thm:intro-existence}

\begin{enumerate}
    \item (Proposition \ref{prop:degree-criterion})
If $\mathcal D_{r,t}([\omega])<t\Lambda_n,$ then $[\omega]$ contains a metric of constant $\mu^{r,t}$-curvature. In particular, if $\mathcal D_{r,t}\le 0$, the constant $\mu^{r,t}$-curvature representative is unique up to scale. 

\item (Theorem \ref{thm:large-r} and Corollary \ref{cor:special-Hermitian}) If $n\geq 3,$ and $\omega$ is not LCK, then the constant $\mu^{r,t}$ problem is solvable for all large $r$. This applies, in particular, to balanced non-K\"ahler, first-Gauduchon non-K\"ahler, and SKT non-K\"ahler metrics.
\end{enumerate}
\end{introthm}

%The degree criterion, the large-$r$ consequence, and the local Aubin criterion are proved respectively in Proposition~3.6, Theorem~3.8 with Corollary~3.9, and Theorem~3.10 with Corollary~3.12.

The LCK case, for which the \(r\)-dependence vanishes, and the Kähler case are treated separately in Propositions \ref{prop:LCK-constant} -- \ref{thm:Kahler-constant} and Corollary \ref{cor:strict-LCK}.

We analyze the global dependence of the deformed Yamabe invariant on $r.$ For this, set
\[
 \mathcal T([\omega])
 :=\left(\int_M|\mathring T|^{2n}\,\mathrm dV_g\right)^{1/n},
\] which is conformally invariant.
\begin{introthm}\label{thm:intro-global-r} For every fixed $t>0$, the following hold.
\begin{enumerate}
    \item (Theorem~\ref{thm:r-global}) The map
\[
 r\longmapsto\lambda_{r,t}([\omega])
\]
is finite, nonincreasing, concave, and globally Lipschitz, with sharp Lipschitz constant
\[
 \frac{n-1}{2}\,\mathcal T([\omega]).
\]
Its asymptotic slopes are
\[
 \lim_{r\to+\infty}\frac{\lambda_{r,t}([\omega])}{r}
 =-\frac{n-1}{2}\,\mathcal T([\omega]),
 \qquad
 \lim_{r\to-\infty}\frac{\lambda_{r,t}([\omega])}{r}=0.
\]

\item (Corollary~\ref{cor:lck-slope}) Consequently, for $n\ge3$,
\[
 \omega\text{ is LCK}
 \quad\Longleftrightarrow\quad
 \lim_{r\to+\infty}\frac{\lambda_{r,t}([\omega])}{r}=0.
\]
Thus local conformal K\"ahlerness is detected variationally by the asymptotic behavior of the deformed Yamabe invariant.
\end{enumerate}
\end{introthm}

The next results concern the critical threshold.
\begin{introthm}\label{thm:intro-phase}
Assume $\mathring T\not\equiv0$ and that $\lambda_{r,t}([\omega]) = t\Lambda_n$ for at least one $r\in \R$. 

\begin{enumerate}
    \item (Theorem~ \ref{thm:saturation-halfline} and Theorem~\ref{thm:endpoint-attainment}) Then there is a unique finite number $r^{\sharp}=r^{\sharp}(t,[\omega])$ such that
\[
 \{r:\lambda_{r,t}([\omega])=t\Lambda_n\}
 =(-\infty,r^{\sharp}].
\]
The endpoint $r^{\sharp}$ has an exact variational characterization, and the critical infimum is not attained for any $r<r^{\sharp}$. At the endpoint $r^{\sharp}$, if $$(2n-1-t)\mu^{1,0}
-\frac{(2n-1)(n-1)}{2}(r^{\sharp}-1)|\tf|^2 > 0$$ pointwise, then the critical level is attained by a smooth positive minimizer.

\item (Theorem \ref{thm:phase}) There is, moreover, a second canonical connection parameter $r_0>r^{\sharp}$ such that
\[
 \lambda_{r,t}([\omega])=
 \begin{cases}
 t\Lambda_n, & r\le r^{\sharp},\\
 \text{a number in }(0,t\Lambda_n), & r^{\sharp}<r<r_0,\\
 0, & r=r_0,\\
 <0, & r>r_0.
 \end{cases}
\]
The invariant is strictly decreasing on $(r^{\sharp},\infty)$.
\end{enumerate}
\end{introthm}

These results separate three phenomena that need not coincide: equality with the spherical level, attainment of that level, and positivity of the quadratic form. Let $B_0(g)$ be the optimal second Sobolev constant in \eqref{eq:sharp-compact}.
\begin{introthm}(Theorem \ref{thm:chern-threshold}).\label{thm:intro-tstar}
Suppose that on a fixed Hermitian background the Chern slice has constant potential
\[
 \mu^{1,t}=A+bt,
 \qquad A>0,\quad b\ge0.
\]
Then the critical value
\[
 t_*:=\frac{A}{\Lambda_n B_0(g)-b}
\]
belongs to $(0,2n-1)$ and satisfies
\[
 \lambda_{1,t}=t\Lambda_n\quad(0<t\le t_*),
 \qquad
 \lambda_{1,t}<t\Lambda_n\quad(t>t_*).
\]
The spherical level is not attained for $t<t_*$ and is attained at $t=t_*$. At the endpoint the minimizers are exactly the extremals for the optimal second Sobolev inequality. The map $t\mapsto\lambda_{1,t}$ is positive, nondecreasing, concave, and continuous, while $\lambda_{1,t}/t$ is constant on $(0,t_*]$ and strictly decreasing after the threshold.
\end{introthm}

In Section \ref{sec:examples}, these critical thresholds are characterized in terms of $B_0(g)$ and applied to $\mathbb{CP}^n$ with the Fubini--Study metric, Hopf manifolds, and the standard Iwasawa manifold and its products. The Iwasawa and Hopf examples show that critical equality and attainment also occur in non-K\"ahler settings.

Section \ref{sec:prescribed} studies the prescribed deformed $\mu^{r,t}$ curvature problem in the positive case.

\begin{introthm} (Theorem \ref{thm:weighted-threshold} and Corollaries \ref{cor:weighted-tests} and \ref{cor:prescribed-window})\label{thm:intro-prescribed}
Assume $\lambda_{r,t}([\omega])>0$, and let $f\in C^\infty(M)$ satisfy $\max_M f>0$. Consider the weighted critical quotient $\lambda_f^{r,t}$ associated with
\[
 \frac{2t}{n-1}\Delta_\omega u+\mu^{r,t}(\omega)u
 =f u^{N-1}.
\]
If 
\[ \lambda_f^{r,t} <
 \frac{t\Lambda_n}{(\max_M f)^{2/N}},
\]
then $f$ is realized as the $\mu^{r,t}$-curvature of a conformal metric. In particular, when $n\ge3$, a sufficient local condition at a positive maximum point $P$ of $f$ is
\[
 (2n-1)\mu^{r,t}(P)-t\,\operatorname{scal}(g)(P)
 +t(n-2)\frac{\Delta f(P)}{f(P)}<0.
\]
Combining this criterion with the exact positivity endpoint in the $r$-parameter yields sufficient parameter intervals of connections on which a given sign-changing function can be prescribed.
\end{introthm}
\bigskip

\noindent \textbf{Acknowledgements.} The author thanks Daniele Angella for many stimulating discussions and suggestions, and Carlo Scarpa for comments on an earlier version.
\section{Hermitian geometry and Analytic Preliminaries}\label{sec:geometry}

Throughout, $(M^{2n},J,\omega)$ is a compact connected complex manifold, $n=\dim_\C M\geq2$, and $g$ is a Hermitian metric. We set
\[
 \omega(X,Y)=g(JX,Y),\qquad
 \omega=\sqrt{-1}\,h_{i\bar j}\dd z^i\wedge\dd\bar z^j,
 \qquad \dv=\frac{\omega^n}{n!}.
\]
All conformal classes are taken with the complex structure $J$ fixed:
\[
 [\omega]=\{u^{2/(n-1)}\omega:u\in C^\infty(M),\ u>0\}.
\]
Every positive conformal factor preserves Hermitian compatibility. Thus the underlying Riemannian conformal class identifies with this Hermitian conformal class, and the classical Yamabe theorem applies without changing $J$.
The Laplacian is $\Delta=\dstar\dd=-\operatorname{div}\nabla$. The Lee form is defined by
\begin{equation}\label{eq:lee}
 \dd\omega^{n-1}=\theta\wedge\omega^{n-1}.
\end{equation}
For real forms, $J\alpha(X_1,\ldots,X_k)=(-1)^k\alpha(JX_1,\ldots,JX_k)$.

Let $T$ be the Chern torsion. Its nonzero holomorphic components are
\begin{equation}\label{eq:chern-components}
 T^k_{ij}=h^{k\bar\ell}(\partial_i h_{j\bar\ell}-\partial_j h_{i\bar\ell}),
 \qquad \theta_i=\sum_k T^k_{ik}.
\end{equation}
We use the full real tensor norm. In a frame with $h_{i\bar j}=\delta_{ij}$,
\begin{equation}\label{eq:norm-convention}
 \abs{T}^2=2\sum_{i,j,k}\abs{T^k_{ij}}^2,
 \qquad \abs{\theta}^2=2\sum_i\abs{\theta_i}^2.
\end{equation}

For Gauduchon's canonical family \cite{Gauduchon1997}, we use
\begin{equation}\label{eq:real-connection}
 g(\nabla^r_XY,Z) =g(\nabla^g_XY,Z)+\frac{r-1}{4}J\dd\omega(X,Y,Z)\notag -\frac{1+r}{4}\dd\omega(JX,Y,Z).
\end{equation}
Let \(r,t\in\R\), and set $A_r:=1+(n-1)r$ and $K_r:=3r^2-2r+1.$ Note that $K_r>0$ for all $r\in\R$, since $ K_r = 3\left(r - \frac{1}{3} \right)^2 + \frac{2}{3}.$ We fix the curvature convention; by \cite[Proposition 2.1]{Barbaro2023} the Ricci form of the \(r\)-Gauduchon connection satisfies
\[
\Ric^r(\omega)
=
\frac{r-1}{2}\dd\dd^*\omega
-
\ii\partial\bar\partial\log\omega^n.
\]
We use the normalization of \cite{APSSW2026}. Thus $\scal^r(\omega) = 2\tr_\omega\left(\Ric^r(\omega)\right)^{1,1}.$ With this normalization, \cite[Eq.\ (33)]{Gauduchon1984} implies that the usual \emph{Riemannian scalar curvature} $\mathrm{scal}(g)$ is related to the Chern scalar curvature by
\begin{equation} \label{eq:riemann-chern}
\mathrm{scal}(g) = \mathrm{scal}^{\mathrm{Ch}}(\omega) +d^*\theta -\tfrac14|T|^2 \,\, .
\end{equation}

Two identities will be used repeatedly.

\begin{lemma}\label{lem:Gauduchon-identities}
For every $r\in\R$,
\begin{align}
\scal^r
&=\scal^{\Ch}+(r-1)\bigl(\dd^*\theta+|\theta|^2\bigr),
\label{eq:scalar-r}\\
|T^r|^2
&=\frac{K_r}{2}|T|^2,
\qquad K_r=3r^2-2r+1.
\label{eq:torsion-r}
\end{align}
\end{lemma}

\begin{proof}
Formula \eqref{eq:scalar-r} follows by tracing the Ricci-form identity for Gauduchon connections; see \cite[Proposition 2.1]{Barbaro2023} and compare \cite{Gauduchon1984}. For \eqref{eq:torsion-r}, choose unitary holomorphic coordinates at a point. The $(2,0)$ part of $T^r$ is $rT$, while the mixed component is $(1-r)P/2$, where, after lowering an index, $P_{\bar i j\bar k}=\overline{T_{ik\bar j}}$. Consequently the mixed component has squared norm $(1-r)^2|T|^2/2$, whereas the pure component contributes $r^2|T|^2$. Hence
\[
|T^r|^2
=\left(r^2+\frac{(1-r)^2}{2}\right)|T|^2
=\frac{3r^2-2r+1}{2}|T|^2.
\]
\end{proof}
In particular, $ \mathrm{scal}^{r}(\omega) = \mathrm{scal}^{\mathrm{Ch}}(\omega) = \mathrm{scal}(g)$ on a K\"ahler manifold. 

Let $T^k_{ij}$ denote the Chern torsion in a holomorphic frame unitary at a point, and let $\theta_i=T^k_{ik}$. Define the trace-free Chern torsion by
\begin{equation}\label{eq:Tcircle-def}
(\tf)^k_{ij}
:=
T^k_{ij}
-\frac1{n-1}\bigl(\delta^k_j\theta_i-\delta^k_i\theta_j\bigr).
\end{equation}
We use the symbol $\tf$ in order to avoid confusing the trace-free tensor with the torsion $T^0$ of the $r=0$ Gauduchon connection.

\begin{lemma}[Orthogonal torsion decomposition]\label{lem:trace-decomposition}
The decomposition \eqref{eq:Tcircle-def} is orthogonal and
\begin{equation}\label{eq:orthogonal}
|T|^2
=|\tf|^2+\frac{2}{n-1}|\theta|^2.
\end{equation}
Equivalently,
\begin{equation}\label{eq:Phi-Tcircle}
|\theta|^2-\frac{n-1}{2}|T|^2
=-\frac{n-1}{2}|\tf|^2\le0.
\end{equation}
\end{lemma}

\begin{proof}
The second summand in \eqref{eq:Tcircle-def} is the $U(n)$-trace component of the Chern torsion. A direct contraction in a unitary frame gives zero cross term and squared norm $2|\theta|^2/(n-1)$, which proves \eqref{eq:orthogonal}. Formula \eqref{eq:Phi-Tcircle} follows immediately.
\end{proof}

The trace-free component has a simple conformal transformation law.

\begin{lemma}\label{lem:Tcircle-conformal}
If $\widehat\omega=e^f\omega$, then
\begin{equation}\label{eq:Tcircle-conformal}
\widehat{\tf}=\tf,
\qquad
|\widehat{\tf}|_{\widehat\omega}^2=e^{-f}|\tf|_\omega^2.
\end{equation}
\end{lemma}

\begin{proof}
Under $h_{i\bar j}\mapsto e^fh_{i\bar j}$, the Chern torsion and its trace transform as
\[
\widehat T^k_{ij}
=T^k_{ij}+\delta^k_j f_i-\delta^k_i f_j,
\qquad
\widehat\theta_i=\theta_i+(n-1)f_i.
\]
Substitution in \eqref{eq:Tcircle-def} cancels the $\dd f$ terms, proving tensorial invariance. The norm of a vector-valued two-form has conformal weight $-1$, giving the second identity.
\end{proof}

\subsection{Definition and exact reduction to the Chern slice}

\begin{definition}\label{def:mu}
For $r,t\in\R$, the $(r,t)$-deformed Gauduchon scalar curvature is
\begin{equation}\label{eq:mu-definition}
\mu^{r,t}(\omega)
:=
\scal^r(\omega)
+\frac{t-A_r}{n-1}\dd^*\theta
+\frac{t-2A_r}{2K_r}|T^r|^2.
\end{equation}
\end{definition}

At $r=1$ this becomes
\begin{equation}\label{eq:Chern-slice}
\mu^{1,t}
=
\scal^{\Ch}
+\frac{t-n}{n-1}\dd^*\theta
+\frac{t-2n}{4}|T|^2,
\end{equation}
which is the deformation studied in \cite{APSSW2026}. In particular,
\begin{equation}\label{eq:special-t}
\mu^{1,2n-1}=\scal(g),
\end{equation}
by the classical Hermitian scalar-curvature identity of Gauduchon \cite{Gauduchon1984}.

\begin{proposition}[Exact dependence on the connection]\label{prop:reduction}
For all $r,t\in\R$,
\begin{equation}\label{eq:exact-reduction}
\mu^{r,t} = \mu^{1,t} -\frac{n-1}{2}(r-1)|\tf|^2.
\end{equation}
\end{proposition}

\begin{proof}
Insert \eqref{eq:scalar-r} and \eqref{eq:torsion-r} into \eqref{eq:mu-definition}. The coefficient of $\dd^*\theta$ simplifies to $(t-n)/(n-1)$, while
\[
\frac{t-2A_r}{2K_r}|T^r|^2
=\frac{t-2-2(n-1)r}{4}|T|^2.
\]
Subtracting the $r=1$ expression gives
\[
\mu^{r,t}-\mu^{1,t}
=(r-1)\left(|\theta|^2-\frac{n-1}{2}|T|^2\right).
\]
Now apply \eqref{eq:Phi-Tcircle}.
\end{proof}

%\begin{remark}[The Bismut slice]\label{rem:Bismut} At $r=-1$,\[\mu^{-1,t}=\mu^{1,t}+(n-1)|\tf|^2.\]
%Thus, for the pointwise existence criteria developed below, moving from the Chern connection toward the Bismut connection is never favourable; the useful direction is $r>1$ wherever $\tf\ne0$.
%\end{remark}

\subsection{Conformal transformation and the LCK boundary}

The Chern deformation satisfies a critical conformal transformation law \cite{APSSW2026}. Proposition \ref{prop:reduction} and Lemma \ref{lem:Tcircle-conformal} show that the same law holds for every Gauduchon parameter.

\begin{proposition}\label{prop:conformal-law}
Let $f\in C^\infty(M)$ and $\widehat\omega=e^f\omega$. Then
\begin{equation}\label{eq:conformal-f}
\mu^{r,t}(\widehat\omega)
=e^{-f}\left(
\mu^{r,t}(\omega)
+t\Delta_\omega f
-\frac{t(n-1)}{2}|\dd f|_\omega^2
\right).
\end{equation}
Equivalently, if $u>0$ and $\widehat\omega=u^{2/(n-1)}\omega$, then
\begin{equation}\label{eq:conformal-u}
\mu^{r,t}(\widehat\omega) =u^{-(N-1)} \left( \frac{2t}{n-1}\Delta_\omega u +\mu^{r,t}(\omega)u \right), \qquad N=\frac{2n}{n-1}.
\end{equation}
\end{proposition}

\begin{proof}
For $r=1$, \eqref{eq:conformal-f} is \cite[Lemma 2.1 and Proposition 3.2]{APSSW2026}. By \eqref{eq:exact-reduction} and \eqref{eq:Tcircle-conformal}, the correction term $-(n-1)(r-1)|\tf|^2/2$ transforms simply by multiplication by $e^{-f}$, so the same formula holds for arbitrary $r$. The substitution $u=e^{(n-1)f/2}$ gives \eqref{eq:conformal-u}.
\end{proof}

\begin{remark}\label{rem:t-zero}
When $t=0$, \eqref{eq:conformal-f} reduces to
\[
\mu^{r,0}(e^f\omega)=e^{-f}\mu^{r,0}(\omega).
\]
Hence the pointwise sign of $\mu^{r,0}$ is a conformal invariant, even though the corresponding Yamabe equation loses its second-order term.
\end{remark}

For \(n\ge3\), vanishing of \(\tf\) is equivalent to the LCK condition.

\begin{proposition}[The LCK boundary]\label{prop:LCK-boundary}
Let $n\ge3$. Then
\begin{equation}\label{eq:LCK-equivalence}
\tf\equiv0
\quad\Longleftrightarrow\quad
\omega\text{ is locally conformally K\"ahler}.
\end{equation}
Consequently, on an LCK conformal class the functions $\mu^{r,t}$ are independent of $r$.
\end{proposition}

\begin{proof}
The identity $\tf=0$ is equivalent, after lowering the upper index of the Chern torsion, to
\[
T_{ij\bar k}
=\frac1{n-1}\left(h_{j\bar k}\theta_i-h_{i\bar k}\theta_j\right).
\]
Equivalently,
\[
\partial\omega=\frac1{n-1}\theta^{1,0}\wedge\omega,
\]
and therefore
\begin{equation}\label{eq:domega-LCK}
\dd\omega=\frac1{n-1}\theta\wedge\omega.
\end{equation}
Applying $\dd$ once more gives $\dd\theta\wedge\omega=0$. For $n\ge3$, the Lefschetz map $\alpha\mapsto\alpha\wedge\omega$ is injective on two-forms, hence $\dd\theta=0$. Thus \eqref{eq:domega-LCK} is the LCK equation. The converse follows by reversing the coefficient comparison. The final assertion is immediate from \eqref{eq:exact-reduction}.
\end{proof}
\subsection{Analytic preliminaries}\label{sec:analytic}
Let
\begin{equation}\label{eq:Lambda-K}
 \Lambda_n:=2n\,\Vol(S^{2n})^{1/n},
 \qquad
 K_{2n}^2:=\frac{2}{(n-1)\Lambda_n}.
\end{equation}
Then $K_{2n}$ is the Euclidean optimal $H^1\to L^N$ Sobolev constant in real dimension $2n$, in the convention
\[
 \|u\|_{L^N(\R^{2n})}^2
 \le K_{2n}^2\|\nabla u\|_{L^2(\R^{2n})}^2;
\]
see \cite{Aubin1976Sobolev,Talenti1976}. Notice that
\begin{equation}\label{eq:at-K}
 \frac{2t}{n-1}=t\Lambda_nK_{2n}^2.
\end{equation}

On a compact manifold, the first optimal compact Sobolev inequality of Hebey--Vaugon \cite{HebeyVaugon1996} says that there is a finite smallest constant $B_0(g)$ such that
\begin{equation}\label{eq:sharp-compact}
 \left(\int_M\abs u^NdV_g\right)^{2/N}
 \le K_{2n}^2\int_M|\nabla u|^2dV_g
 +B_0(g)\int_Mu^2dV_g
\end{equation}
for every $u\in H^1(M)$.
We use \cite{Aubin1976Sobolev,Talenti1976,HebeyVaugon1996,DruetHebey2002} for these sharp inequalities.
Testing constants gives
\begin{equation}\label{eq:B0-volume}
 B_0(g)\ge \Vol(M,g)^{-1/n}.
\end{equation}
Equivalently,
\begin{equation}\label{eq:B0-sup}
 B_0(g)
 =\sup_{u\ne0}
 \frac{\left(\int_M\abs u^N\right)^{2/N}-K_{2n}^2\int|\nabla u|^2}
 {\int u^2}.
\end{equation}

%We shall use the following classical analytic facts.

%\begin{theorem}[Aubin's critical upper bound]\label{thm:aubin-input}
%For every smooth potential $V$ on a compact real $2n$-manifold and every $a_t>0$, the critical quotient
%\[\inf_{u\ne0}\frac{\int(a_t|\nabla u|^2+Vu^2)}{D(u)}\]
%is at most $a_t/K_{2n}^2=t\Lambda_n$.  If the inequality is strict, the infimum is achieved by a smooth positive function.
%\end{theorem}

%\begin{proof}[Reference and normalization]
%This is the standard critical compactness theorem underlying the Yamabe problem; see \cite{Aubin1976Yamabe,Aubin1998,DruetHebey2002}.  Local Euclidean bubbles give the upper bound $a_t/K_{2n}^2$.  If the infimum is strictly smaller, a normalized minimizing sequence cannot lose mass by concentration, because every concentration point costs at least the Euclidean level.  Strong $H^1$ compactness follows, and elliptic regularity plus the strong maximum principle gives a smooth positive minimizer.  Identity \eqref{eq:at-K} converts the Euclidean level to $t\Lambda_n$.
%\end{proof}

%Applied to $V=\mu^{r,t}(\omega)$, this gives
%\begin{equation}\label{eq:sharp-upper} \lambda_{r,t}([\omega])\le t\Lambda_n.
%\end{equation}

We also use the Djadli--Druet extremal theorem \cite{DjadliDruet2001}.  In the present normalization, if
\begin{equation}\label{eq:DD}
 B_0(g)>
 \frac{2n-2}{4(2n-1)}K_{2n}^2\max_M\scal(g),
\end{equation}
then equality in \eqref{eq:sharp-compact} is attained by a nonzero $H^1$ function.  In particular, an extremal exists whenever $\scal(g)\le0$ on $M$; see also \cite{DruetHebey2002}.

\subsection{Critical functions and the endpoint theorem of Collion}
For a smooth function $h$ on a compact real $m$-manifold, $m\ge4$, define
\begin{equation}\label{eq:critical-h-functional}
 \lambda(h)
 :=\inf_{u\ne0}
 \frac{\int_M(|\nabla u|^2+hu^2)dV_g}
 {\left(\int_M|u|^{2m/(m-2)}dV_g\right)^{(m-2)/m}}.
\end{equation}
Following Hebey--Vaugon and Collion \cite{HebeyVaugon2001,Collion2007}, $h$ is called weakly critical when
\[
 \lambda(h)=K_m^{-2}.
\]
We use the following two consequences of this theory.

\begin{theorem}\label{thm:collion-input}
Let $m > 4$.
\begin{enumerate}[label=\textup{(\roman*)}]
\item If $h$ is weakly critical, then
\begin{equation}\label{eq:collion-necessary}
 \frac{4(m-1)}{m-2}h\ge \scal(g)
 \qquad\text{pointwise on }M.
\end{equation}
\item Suppose $h$ is weakly critical, $\Delta_g+h$ is coercive, and
\begin{equation}\label{eq:collion-strict}
 \frac{4(m-1)}{m-2}h>\scal(g)
 \qquad\text{on }M.
\end{equation}
Assume that there is a family $h_s\le h$, $h_s\not\equiv h$, converging to $h$ in $C^{0,\alpha}$ for some $\alpha\in(0,1)$, such that $h_s$ is subcritical and $\Delta_g+h_s$ is coercive for all $s$ sufficiently close to the endpoint.  Then the weakly critical quotient for $h$ admits a smooth positive minimizer.
\end{enumerate}
\end{theorem}

\begin{proof}[Reference]
Part (i) is the local necessary condition obtained by testing a weakly critical potential with concentrating bubbles; it is Proposition 1 of \cite{Collion2007} in the constant-weight case.  Part (ii) is the constant-weight specialization of Theorem 1 in \cite{Collion2007}.  These are the hypotheses used later: a weakly critical endpoint, a subcritical family approaching it from below, coercivity, and the strict pointwise scalar-curvature barrier.
\end{proof}

\section{The constant-curvature problem}
By \eqref{eq:conformal-u}, \(\mu^{r,t}(\widetilde\omega)=\lambda\) is equivalent to solving
\begin{equation}\label{eq:rt-DY}
\frac{2t}{n-1}\Delta u + \mu^{r,t}(\omega)u = \lambda u^{\frac{n+1}{n-1}}.
\end{equation}
For \(t>0\), \eqref{eq:rt-DY} is a semilinear elliptic equation. For \(t=0\), the second-order term disappears, and the problem degenerates.

\subsection{Yamabe-type functional and invariant}

Fix \(t>0\). Define
\begin{equation}\label{eq:functional}
\mathcal F^{r,t}(u)
=
\frac{
\displaystyle
\int_M
\left(
\frac{2t}{n-1}|\nabla u|^2
+
\mu^{r,t}(\omega)u^2
\right)\frac{\omega^n}{n!}
}
{
\displaystyle
\left(
\int_M u^N\,\frac{\omega^n}{n!}
\right)^{2/N}
},
\quad u\in H^1(M)\setminus\{0\},\;\; 
u>0.
\end{equation}
The proof of \cite[Theorem 3.10]{APSSW2026}, together with Proposition \ref{prop:conformal-law} gives the following variational characterization.
\begin{theorem} \label{thm:variational} The equation \eqref{eq:rt-DY} is the Euler--Lagrange equation of the $(r,t)$-deformed Einstein--Hilbert functional $\mathcal F^{r,t}$ restricted to the conformal class of $\omega$ with appropriate $\lambda$.
\end{theorem}

\begin{proposition}[\cite{Yamabe1960}; see also \cite{LeeParker1987,APSSW2026}] \label{prop:minimizer}
Let $(M^{2n},J,\omega)$ be a compact Hermitian manifold of complex dimension $n$. For $2 \leq p \leq N$, consider the functional
$$\mathcal F_p^{r,t}(u)
=
\frac{
\displaystyle
\int_M
\left(
\frac{2t}{n-1}|\nabla u|^2
+
\mu^{r,t}(\omega)u^2
\right)\frac{\omega^n}{n!}
}
{
\displaystyle
\left(
\int_M u^p\,\frac{\omega^n}{n!}
\right)^{2/p}
}, \quad u\in H^1(M)\setminus\{0\},\;\; 
u>0 $$
Then the functional $\mathcal F_p^{r,t}$ is bounded from below. Moreover, when the exponent $p$ satisfies
$$
2 \leq p < N = \frac{2n}{n-1} \,\, ,
$$
the functional $\mathcal F_p^{r,t}$ admits a minimum.
\end{proposition}

\begin{definition}\label{def:lambda}
The \((r,t)\)-Gauduchon deformed Yamabe invariant of the Hermitian conformal
class \([\omega]\) is
\begin{equation}\label{eq:lambda}
\lambda_{r,t}([\omega]) :=\inf_{u\in H^1(M)\setminus\{0\}}
\mathcal F^{r,t}(u).
\end{equation}
\end{definition}
The Euler--Lagrange equation of \(\mathcal F^{r,t}\) is precisely
\eqref{eq:rt-DY}. Hence minimizers give metrics of constant
\(\mu^{r,t}\)-curvature.

%Define
%\begin{equation}\label{eq:Qrt}
%a_t:=\frac{2t}{n-1},\quad \energy_{r,t}(u)
%:=
%\int_M\left(a_t|\nabla u|^2+\mu^{r,t}(\omega)u^2\right)\,\mathrm dV_\omega
%\end{equation}
%and
%\begin{equation}\label{eq:lambda}
%\lambda_{r,t}([\omega])
%:=
%\inf_{u\in H^1(M)\setminus\{0\}}
%\frac{\energy_{r,t}(u)}{\norm{u}_{L^N}^2}.
%\end{equation}
Conformal transformation shows that \eqref{eq:lambda} depends only on the Hermitian conformal class.

Henceforth, for simplicity, we write
\begin{equation}\label{eq:QD}
 a_t:=\frac{2t}{n-1},\quad \energy_{r,t}(u)=\int_M(a_t\abs{\nabla u}^2+\mu^{r,t}u^2)\frac{\omega^n}{n!},
 \quad D(u)=\left(\int_M\abs u^N\frac{\omega^n}{n!}\right)^{2/N},
\end{equation}
\begin{equation}\label{eq:QD2}
\lambda_{r,t}= \lambda_{r,t}([\omega])=\inf_{u\in H^1(M)\setminus\{0\}}\frac{\energy_{r,t}(u)}{D(u)},\quad W = \abs{\tf}^2,\; c_n = \frac{n-1}{2}.
\end{equation}
Unless a metric is displayed, the background is $g$. We write $\ty=2n-1$ for the Riemannian value on the Chern slice. Notice the exact normalization $a_t=t\Lambda_n K_{2n}^2$.

\begin{lemma}\label{lem:coercivity}
If $\lambda_{r,t} = \lambda_{r,t}([\omega])>0$, there is $C>0$ such that $\energy_{r,t}(u)\geq C\norm u_{H^1}^2$ for every $u\in H^1(M)$.
\end{lemma}
\begin{proof}

Put $h:=\mu^{r,t}(\omega)$ and $\mathcal V:=\Vol(M,\omega)$.  By definition of $\lambda_{r,t}$,
\begin{equation}\label{eq:Y-controls-N}
  \energy_{r,t}(u)\geq \lambda_{r,t}\|u\|_N^2.
\end{equation}
Hölder's inequality gives
\[
  \|u\|_2^2\leq \mathcal V^{1-2/N}\|u\|_N^2
  =\mathcal V^{1/n}\|u\|_N^2,
\]
so
\begin{equation}\label{eq:L2-controlled}
  \|u\|_2^2\leq \frac{\mathcal V^{1/n}}{\lambda_{r,t}}\energy_{r,t}(u).
\end{equation}
Let $h^-:=\max\{-h,0\}$.  Since $h=h^+-h^-$,
\begin{align*}
  a_t\|\nabla u\|_2^2
  &=\energy_{r,t}(u)-\int_Mhu^2\frac{\omega^n}{n!}\\
  &=\energy_{r,t}(u)-\int_Mh^+u^2\frac{\omega^n}{n!}
    +\int_Mh^-u^2\frac{\omega^n}{n!}\\
  &\leq \energy_{r,t}(u)+\|h^-\|_{\infty}\|u\|_2^2.
\end{align*}
Using \eqref{eq:L2-controlled},
\[
  \|\nabla u\|_2^2
  \leq \frac1{a_t}\left(1+
  \frac{\mathcal V^{1/n}\|h^-\|_{\infty}}{\lambda_{r,t}}\right)\energy_{r,t}(u).
\]
Together with \eqref{eq:L2-controlled}, this yields
$\|u\|_{H^1}^2\leq c_{r,t}\energy_{r,t}(u)$, proving the assertion.
\end{proof}

\subsection{Sharp threshold and a deformed degree}

%Let
%\begin{equation}\label{eq:Lambda}
%\Lambda_n:=2n\Vol(S^{2n})^{1/n}.
%\end{equation}
The following is the usual Aubin compactness statement for a critical quotient with smooth potential.

\begin{theorem}[Sharp threshold]\label{thm:sharp-threshold}
For every $r\in\R$ and $t>0$,
\begin{equation}\label{eq:sharp-upper}
\lambda_{r,t}([\omega])\le t\Lambda_n.
\end{equation}
If the inequality is strict, then the infimum in \eqref{eq:lambda} is attained by a smooth positive function. Consequently $[\omega]$ contains a metric of constant $\mu^{r,t}$-curvature.
\end{theorem}

\begin{proof}
Divide the quotient by $a_t$. It becomes the standard critical Sobolev quotient for an operator $\Delta+H$ with smooth potential $H=\mu^{r,t}/a_t$ on the real $2n$-manifold $(M,g)$. Local Aubin bubbles yield the upper bound, and the strict inequality prevents concentration and gives compactness; see \cite{Aubin1998,LeeParker1987}. Multiplying the sharp constant by $a_t=2t/(n-1)$ gives $t\Lambda_n$.
\end{proof}

In particular, $\lambda_{r,t}([\omega])\le0$ automatically lies below the concentration level. In this non-positive case, the usual maximum-principle argument also gives uniqueness of the constant $\mu^{r,t}$-curvature representative up to constant rescaling.

Let $\eta$ denote the unique volume-one Gauduchon metric in the conformal class, so that $\dd\dd^c\eta^{n-1}=0$, equivalently $\dd^*_{\eta}\theta_\eta=0$ \cite{Gauduchon1977,Gauduchon1984}. Define
\[
 \Gamma^r=\int_M\scal^r(\eta)\frac{\eta^n}{n!}.
\]
These quantities use the scalar and torsion conventions of Section~\ref{sec:geometry}.
\begin{definition}\label{def:deformed-degree}
The \((r,t)\)-deformed Gauduchon degree is
\begin{equation}\label{eq:degree}
    \degree^{r,t}([\omega]) :=\int_M\mu^{r,t}(\eta)\frac{\eta^n}{n!}
   =\Gamma^r+\frac{t-2A_r}{4}\int_M\abs{T_\eta}^2\frac{\eta^n}{n!}.
\end{equation}
\end{definition}

\begin{proposition}[Deformed degree criterion]\label{prop:degree-criterion}
If
\begin{equation}\label{eq:degree-condition}
\degree^{r,t}([\omega])<t\Lambda_n,
\end{equation}
then $[\omega]$ contains a metric of constant $\mu^{r,t}$-curvature. If $\degree^{r,t}\leq0$, it is unique up to scale and has nonpositive curvature. Moreover,
\begin{equation}\label{eq:degree-r}
\degree^{r,t}([\omega])
=
\degree^{1,t}([\omega])
-\frac{n-1}{2}(r-1)
\int_M|\tf_\eta|^2\,\mathrm dV_\eta.
\end{equation}
A sufficient special case is $0<t\leq2A_r,$ and $\Gamma^r\leq0;$ more generally, the criterion allows positive $\Gamma^r$ whenever
\begin{equation}\label{eq:degree-positive}
 \Gamma^r<t\Lambda_n+\frac{2A_r-t}{4}\int_M\abs{T_\eta}^2\frac{\eta^n}{n!}.
\end{equation}
\end{proposition}

\begin{proof}
Since $\Vol(M,\eta)=1$, the test function $u\equiv1$ gives
\[
\lambda_{r,t}([\omega])\le \degree^{r,t}([\omega]).
\]
Thus \eqref{eq:degree-condition} implies the strict inequality in Theorem \ref{thm:sharp-threshold}. Formula \eqref{eq:degree-r} follows by integrating \eqref{eq:exact-reduction} for $\eta$.
\end{proof}

\begin{remark}\label{rem:degree-sign}
At $t=2A_r$, the degree equals $\Gamma^r$ whenever this value of $t$ lies in the positive range. In particular $\degree^{1,2n}=\Gamma^\Ch$, which is twice the Chern degree normalized in \cite{ACS2017}. This equality of degree values does not identify the deformed equation with the undeformed Chern--Yamabe equation: a divergence term remains away from the Gauduchon representative.
\end{remark}

The degree criterion gives the following large-\(r\) existence result.

\begin{theorem}[Large connection parameter]\label{thm:large-r}
Let $t>0$. If $\tf\not\equiv0$ in the conformal class $[\omega]$, then there exists $R\in\R$ such that for every $r>R$ the class $[\omega]$ contains a metric of constant $\mu^{r,t}$-curvature.

If $n\ge3$, the hypothesis is equivalent to saying that $\omega$ is not locally conformally K\"ahler.
\end{theorem}

\begin{proof}
The property $\tf\equiv0$ is conformally invariant by Lemma \ref{lem:Tcircle-conformal}. Hence $\tf_\eta\not\equiv0$, so the integral in \eqref{eq:degree-r} is strictly positive. Therefore $\degree^{r,t}([\omega])\to-\infty$ as $r\to+\infty$, and Proposition \ref{prop:degree-criterion} applies for all sufficiently large $r$. The final statement follows from Proposition \ref{prop:LCK-boundary}.
\end{proof}
Recall
that a Hermitian metric $\omega$ is called \emph{first-Gauduchon} if $\partial\bar\partial\omega\wedge\omega^{n-2}=0,$ while it is called \emph{strong K\"ahler with torsion} (SKT), or
\emph{pluriclosed}, if $\partial\bar\partial\omega=0.$
In particular, every SKT metric is first-Gauduchon; see
\cite{FinoUgarte2013}.
\begin{corollary}\label{cor:special-Hermitian}
Let $n\ge3$ and $t>0$. The conclusion of Theorem \ref{thm:large-r} holds if $\omega$ is any of the following:
\begin{enumerate}
\item balanced and non-K\"ahler;
\item first-Gauduchon and non-K\"ahler;
\item SKT and non-K\"ahler.
\end{enumerate}
\end{corollary}

\begin{proof}
A balanced LCK metric has $\theta=0$, hence is K\"ahler, so case (i) is non-LCK. For a first-Gauduchon metric one has \cite[proof of Theorem 2.5]{FinoUgarte2013}
\[
\dd^*\theta=\frac12|T|^2-|\theta|^2.
\]
If it is also LCK, then $\tf=0$ by Proposition \ref{prop:LCK-boundary}, so \eqref{eq:orthogonal} gives $|T|^2=2|\theta|^2/(n-1)$. Integrating the displayed first-Gauduchon identity yields $\theta\equiv0$ for $n\ge3$, and hence the metric is K\"ahler. This proves (ii). Every SKT metric is first-Gauduchon, giving (iii).
\end{proof}

\subsection{A pointwise Aubin criterion}

The degree criterion is global. A local bubble test can give strict inequality for much smaller values of $r$. Define
\begin{equation}\label{eq:E}
E_{r,t}(\omega)
:=(2n-1)\mu^{r,t}(\omega)-t\scal(g).
\end{equation}

\begin{theorem}[Pointwise Aubin criterion]\label{thm:pointwise-Aubin}
Let $n\ge3$ and $t>0$. If there is $p\in M$ such that
\begin{equation}\label{eq:E-negative}
E_{r,t}(\omega)(p)<0,
\end{equation}
then $\lambda_{r,t}([\omega])<t\Lambda_n.$ In particular, $[\omega]$ contains a metric of constant $\mu^{r,t}$-curvature.
\end{theorem}

\begin{proof}
For an operator $\Delta+H$ on a real $m$-manifold, the standard Aubin test-function expansion gives strict inequality whenever
\[
4\frac{m-1}{m-2}H(p)<\scal(g)(p);
\]
see \cite{Aubin1998, Aubin1976Yamabe} and the adaptation in \cite{APSSW2026}. Here $m=2n$ and $H=\mu^{r,t}/a_t$. Since $a_t=2t/(n-1)$, the condition becomes
\[
\frac{2n-1}{t}\mu^{r,t}(p)<\scal(g)(p),
\]
which is exactly \eqref{eq:E-negative}.
\end{proof}

\begin{proposition}\label{prop:E-reduction}
For all $r,t$,
\begin{equation}\label{eq:E-reduction}
E_{r,t}
=E_{1,t}
-\frac{(2n-1)(n-1)}{2}(r-1)|\tf|^2,
\end{equation} where,
\begin{equation}\label{eq:E-Chern-mu0}
E_{1,t}=(2n-1-t)\mu^{1,0}.
\end{equation}
\end{proposition}

\begin{proof}
Equation \eqref{eq:E-reduction} follows directly from \eqref{eq:exact-reduction}. For \eqref{eq:E-Chern-mu0}, use $\scal(g)=\mu^{1,2n-1}$ from \eqref{eq:special-t} and the fact that $\mu^{1,t}$ is affine in $t$ by \eqref{eq:Chern-slice}.
\end{proof}

\begin{corollary}[Optimizing in $r$]\label{cor:r-optimized}
Let $n\ge3$, $t>0$, and suppose $\tf(p)\ne0$. Then the pointwise Aubin condition holds at $p$ for every
\begin{equation}\label{eq:r-pointwise-threshold}
r>
1+
\frac{2E_{1,t}(p)}{(2n-1)(n-1)|\tf(p)|^2}.
\end{equation}
In particular, wherever the trace-free Chern torsion is nonzero, increasing $r$ eventually forces the strict Aubin inequality.
\end{corollary}

\subsection{LCK classes and K\"ahler starting metrics}\label{subsec:LCK-Kahler}
When \(\tf \equiv0\), the \(r\)-dependence disappears. We treat the resulting LCK case separately and the case of conformal classes containing a Kähler metric.

\begin{proposition}[The constant problem on an LCK class]\label{prop:LCK-constant}
Let $n\ge3$, let $t>0$, and suppose that $\omega$ is LCK. Then, for every $r\in\R$,
\begin{equation}\label{eq:LCK-r-independence}
\mu^{r,t}=\mu^{1,t},
\qquad
\energy_{r,t}=\energy_{1,t},
\qquad
\lambda_{r,t}([\omega])=\lambda_{1,t}([\omega]).
\end{equation}
Let $\eta$ be the unit-volume Gauduchon representative of the conformal class and put
\begin{equation}\label{eq:Gamma-Theta}
\Gamma([\omega])
:=\int_M\scal^{\Ch}(\eta)\,\mathrm dV_\eta,
\qquad
\Theta([\omega])
:=\int_M|\theta_\eta|^2\,\mathrm dV_\eta.
\end{equation}
Then
\begin{equation}\label{eq:LCK-degree}
\degree^{r,t}([\omega])
=\Gamma([\omega])
+\frac{t-2n}{2(n-1)}\Theta([\omega]),
\end{equation}
and
\begin{equation}\label{eq:LCK-E}
E_{r,t}=(2n-1-t)\mu^{1,0}.
\end{equation}
Consequently, a constant $\mu^{r,t}$-curvature metric exists if either
\begin{equation}\label{eq:LCK-degree-criterion}
\Gamma([\omega])
+\frac{t-2n}{2(n-1)}\Theta([\omega])
<t\Lambda_n,
\end{equation}
or there is a point $p\in M$ such that
\begin{equation}\label{eq:LCK-pointwise-criterion}
\bigl(t-(2n-1)\bigr)\mu^{1,0}(p)>0.
\end{equation}
If $\mu^{1,0}$ changes sign, or if $\mu^{1,0}\equiv0$, then the constant $\mu^{r,t}$-curvature problem is solvable for every $t>0$ and every $r\in\R$.
\end{proposition}

\begin{proof}
Because the conformal class is LCK, Proposition~\ref{prop:LCK-boundary} gives $\tf\equiv0$. The exact reduction \eqref{eq:exact-reduction} therefore gives $\mu^{r,t}=\mu^{1,t}$ pointwise for every $r$. Since the gradient coefficient $a_t$ does not depend on $r$, the quadratic forms and their infima are also independent of $r$, proving \eqref{eq:LCK-r-independence}.

The metric $\eta$ is Gauduchon, hence $\dd^*\theta_\eta=0$. Moreover, on an LCK metric the trace-free torsion vanishes, and the orthogonal decomposition \eqref{eq:orthogonal} reduces to
\begin{equation}\label{eq:LCK-T-theta}
|T_\eta|^2=\frac{2}{n-1}|\theta_\eta|^2.
\end{equation}
Substituting these two facts into the Chern-slice formula \eqref{eq:Chern-slice} gives
\[
\mu^{1,t}(\eta)
=\scal^{\Ch}(\eta)
+\frac{t-2n}{2(n-1)}|\theta_\eta|^2.
\]
Integration over the unit-volume representative proves \eqref{eq:LCK-degree}. Formula \eqref{eq:LCK-E} follows from Proposition~\ref{prop:E-reduction}, because the same proposition gives $E_{1,t}=(2n-1-t)\mu^{1,0}$ and $E_{r,t}=E_{1,t}$ in the LCK case. The criteria \eqref{eq:LCK-degree-criterion} and \eqref{eq:LCK-pointwise-criterion} are therefore exactly Proposition~\ref{prop:degree-criterion} and Theorem~\ref{thm:pointwise-Aubin}, respectively.

Assume next that $\mu^{1,0}$ changes sign. If $t>2n-1$, choose $p$ with $\mu^{1,0}(p)>0$; then \eqref{eq:LCK-pointwise-criterion} holds. If $0<t<2n-1$, choose instead a point where $\mu^{1,0}<0$. When $t=2n-1$, one has $\mu^{r,2n-1}=\mu^{1,2n-1}=\scal(g)$ on the whole LCK class, so the equation is precisely the classical Yamabe equation and is solvable by the Yamabe theorem.

It remains to consider $\mu^{1,0}\equiv0$. Since $t\mapsto\mu^{1,t}$ is affine and $\mu^{1,2n-1}=\scal(g)$, we have pointwise on the LCK class
\begin{equation}\label{eq:LCK-affine-identity}
\mu^{r,t}=\mu^{1,t}
=\frac{t}{2n-1}\scal(g)
+\frac{2n-1-t}{2n-1}\mu^{1,0}
=\frac{t}{2n-1}\scal(g).
\end{equation}
The classical Yamabe gradient coefficient in real dimension $2n$ is
\[
a_Y=\frac{4(2n-1)}{2n-2}
=\frac{2(2n-1)}{n-1},
\]
and hence $a_t=(t/(2n-1))a_Y$. Therefore, for every nonzero $u\in H^1(M)$,
\[
\energy_{r,t}(u)
=\frac{t}{2n-1}
\int_M\left(a_Y|\nabla u|^2+\scal(g)u^2\right)\,\mathrm dV_g.
\]
Thus the deformed quotient is exactly $t/(2n-1)$ times the classical Yamabe quotient. A Yamabe minimizer consequently solves the constant $\mu^{r,t}$ equation for every $t>0$. This proves the final assertion.
\end{proof}

\begin{proposition}[Conformal classes containing a K\"ahler metric]\label{thm:Kahler-constant}
Let $n\ge3$, and suppose that $[\omega]$ contains a K\"ahler metric $\kappa$. After a constant rescaling assume $\Vol(M,\kappa)=1$, and write $S_\kappa:=\scal(g_\kappa)$. Then the following hold for every $r\in\R$.
\begin{enumerate}[label=\textup{(\roman*)}]
\item If $t=2n-1$, the constant $\mu^{r,t}$ problem is exactly the classical Yamabe problem and therefore has a solution.
\item If $t>2n-1$, the conformal class contains a metric of constant $\mu^{r,t}$-curvature.
\item If $0<t<2n-1$, a solution exists whenever
\[
\int_M S_\kappa\,\mathrm dV_\kappa\le0
\qquad\text{or}\qquad
S_\kappa(p)<0\text{ for some }p\in M.
\]
If $S_\kappa$ is constant, then $\kappa$ itself is a solution for every $t>0$.
\end{enumerate}
\end{proposition}

\begin{proof}
Because $\kappa$ is K\"ahler, its Lee form and Chern torsion vanish: $\theta_\kappa=0,$ $T_\kappa=0.$
Consequently $\tf_\kappa=0$, so the entire conformal class is LCK and Proposition~\ref{prop:LCK-constant} shows that the problem is independent of $r$. Evaluating Definition~\ref{def:mu} on the K\"ahler representative gives, for every $r$ and $t$,
\begin{equation}\label{eq:Kahler-mu-start}
\mu^{r,t}(\kappa)=S_\kappa.
\end{equation}
Since a K\"ahler metric is Gauduchon and $\kappa$ has unit volume, uniqueness of the normalized Gauduchon representative gives
\begin{equation}\label{eq:Kahler-degree}
\degree^{r,t}([\omega])
=\int_M S_\kappa\,\mathrm dV_\kappa.
\end{equation}
Furthermore, Proposition~\ref{prop:E-reduction} and \eqref{eq:Kahler-mu-start} yield
\begin{equation}\label{eq:Kahler-E}
E_{r,t}
=(2n-1-t)S_\kappa.
\end{equation}

For $t=2n-1$, the LCK $r$-independence and \eqref{eq:special-t} imply
\[
\mu^{r,2n-1}=\scal(g)
\]
for every metric in the conformal class. Thus the equation is exactly the classical Yamabe equation, proving (i).

Assume now that $t>2n-1$. If
\[
\int_M S_\kappa\,\mathrm dV_\kappa\le0,
\]
then \eqref{eq:Kahler-degree} gives $\degree^{r,t}([\omega])\le0<t\Lambda_n,$ and Proposition~\ref{prop:degree-criterion} gives a solution. If instead the total scalar curvature is positive, then $S_\kappa(p)>0$ at some point $p$. Since $2n-1-t<0$, equation \eqref{eq:Kahler-E} gives $E_{r,t}(p)<0$, and Theorem~\ref{thm:pointwise-Aubin} again gives a solution. This proves (ii).

Finally let $0<t<2n-1$. If the total scalar curvature is non-positive, the same degree argument applies. If $S_\kappa(p)<0$ somewhere, then $2n-1-t>0$ and \eqref{eq:Kahler-E} gives $E_{r,t}(p)<0$, so the pointwise Aubin criterion applies. If $S_\kappa$ is constant, \eqref{eq:Kahler-mu-start} shows directly that the starting K\"ahler metric already has constant $\mu^{r,t}$-curvature. This proves (iii).
\end{proof}

\begin{corollary}[LCK classes without K\"ahler representative]\label{cor:strict-LCK}
Suppose that $n\ge3$ and that $[\omega]$ is LCK but contains no K\"ahler metric. Then its unit-volume Gauduchon representative $\eta$ is non-K\"ahler and
\[
\Theta([\omega])=\int_M|\theta_\eta|^2\,\mathrm dV_\eta>0.
\]
For every $r\in\R$ the constant problem is the same, and the sufficient criteria \eqref{eq:LCK-degree-criterion} and \eqref{eq:LCK-pointwise-criterion} apply. In particular, the problem is solvable for all $t>0$ if $\mu^{1,0}$ changes sign or vanishes identically, and it is always solvable at $t=2n-1$ by the classical Yamabe theorem.
\end{corollary}

\begin{proof}
If $\theta_\eta\equiv0$, then the LCK identity \eqref{eq:domega-LCK} gives $\dd\eta=0$, so $\eta$ is K\"ahler, contrary to the hypothesis. Hence $\Theta([\omega])>0$. All remaining assertions are precisely Proposition~\ref{prop:LCK-constant}.
\end{proof}

\begin{remark}\label{rem:LCK-nonKahler-distinction}
An LCK metric may itself be non-K\"ahler while being globally conformal to a K\"ahler metric. In that situation, its conformal class falls under Proposition~\ref{thm:Kahler-constant}. Corollary~\ref{cor:strict-LCK} concerns the case in which the conformal class contains no K\"ahler representative.
\end{remark}

\section{Global dependence on the Gauduchon parameter}\label{sec:rdependence}

Fix $t>0.$ 
%and abbreviate
%\begin{equation}\label{eq:W-c}  W:=|\tf|^2, \qquad c_n:=\frac{n-1}{2}.\end{equation}
The exact reduction \eqref{eq:exact-reduction} gives, for every $r,s\in\R$,
\begin{equation}\label{eq:Q-affine}
 \energy_{r,t}(u)
 =\energy_{s,t}(u)-c_n(r-s)\int_MWu^2dV_g.
\end{equation}
\subsection{A conformal torsion mass}
\begin{definition}\label{def:Tmass}
Define
\begin{equation}\label{eq:Tmass}
 \Tmass([\omega])
 :=\|W\|_{L^n(M,g)}
 =\left(\int_M|\tf|^{2n}dV_g\right)^{1/n}.
\end{equation}
\end{definition}
The quantity $\Tmass([\omega])$ is independent of the representative $\omega\in[\omega]$. Lemma \ref{lem:holder-sharp} below gives the optimal bound for the $L^2$ term in \eqref{eq:Q-affine}.

\begin{lemma}\label{lem:holder-sharp}
For every $u\in H^1(M)$,
\begin{equation}\label{eq:holder-weight}
 \int_MWu^2dV_g\le\Tmass([\omega])D(u).
\end{equation}
Moreover,
\begin{equation}\label{eq:holder-sup}
 \sup_{u\ne0}
 \frac{\int_MWu^2dV_g}{D(u)}
 =\Tmass([\omega]).
\end{equation}
\end{lemma}

\begin{proof}
H\"older with conjugate exponents $n$ and $n/(n-1)$ gives
\begin{align*}
 \int_MWu^2dV_g
 &\le\left(\int_MW^n dV_g\right)^{1/n}
 \left(\int_M|u|^{2n/(n-1)}dV_g\right)^{(n-1)/n}\\
 &=\Tmass([\omega])D(u).
\end{align*}
This proves the upper bound in \eqref{eq:holder-sup}.

If $W\not\equiv0$, equality in H\"older is formally obtained from $u^2=W^{n-1},\qquad u=W^{(n-1)/2}.$ This function belongs to $L^N$ because $W^n\in L^1$.  Approximate it in $L^N$ by smooth functions $u_j$.  Since $W\in L^n$, H\"older also shows
\[
 \int_MW(u_j^2-u^2)dV_g\longrightarrow0,
\]
while $D(u_j)\to D(u)$.  Hence the quotient tends to $\|W\|_{L^n}$.  If $W\equiv0$, both sides of \eqref{eq:holder-sup} are zero.
\end{proof}

\subsection{Concavity, Lipschitz continuity, and asymptotic slope}

\begin{theorem}\label{thm:r-global}
For fixed $t>0$, the function
\[
 r\longmapsto\lambda_{r,t}([\omega])
\]
is finite on $\R$, non-increasing, concave, and globally Lipschitz.  More precisely,
\begin{equation}\label{eq:Lipschitz}
 |\lambda_{r,t}-\lambda_{s,t}|
 \le c_n\Tmass([\omega])|r-s|.
\end{equation}
Furthermore,
\begin{equation}\label{eq:asymptotic-plus}
 \lim_{r\to+\infty}\frac{\lambda_{r,t}}r
 =-c_n\Tmass([\omega]),
\end{equation}
and
\begin{equation}\label{eq:asymptotic-minus}
 \lim_{r\to-\infty}\frac{\lambda_{r,t}}r=0.
\end{equation}
\end{theorem}

\begin{proof}
\textit{Step 1: finiteness.}
Let $\mu_1=\mu^{1,t}(\omega)$.  Since $\mu_1$ and $W$ are smooth on the compact manifold, for each fixed $r$ the potential $\mu^{r,t}=\mu_1-c_n(r-1)W$ is bounded below.  By H\"older,
\[
 \int_Mu^2dV_g\le\Vol(M,g)^{1/n}D(u).
\]
Therefore
\[
 \frac{\energy_{r,t}(u)}{D(u)}
 \ge-\|\mu^{r,t,-}\|_\infty\Vol(M,g)^{1/n},
\]
so the infimum is finite.  The upper bound \eqref{eq:sharp-upper} gives finiteness from above.

\textit{Step 2: monotonicity and Lipschitz continuity.}
For $u\ne0$, put
\[
 F_r(u):=\frac{\energy_{r,t}(u)}{D(u)},
 \qquad
 b(u):=\frac{\int_MWu^2dV_g}{D(u)}.
\]
By \eqref{eq:Q-affine},
\begin{equation}\label{eq:F-affine}
 F_r(u)=F_s(u)-c_n(r-s)b(u).
\end{equation}
Since $b(u)\ge0$, each $F_r(u)$ is non-increasing in $r$, and therefore so is its infimum $\lambda_{r,t}$.  By Lemma \ref{lem:holder-sharp}, $0\le b(u)\le\Tmass$.  Hence
\[
 |F_r(u)-F_s(u)|\le c_n\Tmass|r-s|
\]
uniformly in $u$.  Taking infima first in one direction and then in the other gives \eqref{eq:Lipschitz}.

\textit{Step 3: concavity.}
Let $r=\vartheta r_1+(1-\vartheta)r_2$ with $0\le\vartheta\le1$.  Since $F_r(u)$ is affine in $r$,
\[
 F_r(u)=\vartheta F_{r_1}(u)+(1-\vartheta)F_{r_2}(u).
\]
Thus
\begin{align*}
 \lambda_{r,t}
 &=\inf_u\left[\vartheta F_{r_1}(u)+(1-\vartheta)F_{r_2}(u)\right]\\
 &\ge\vartheta\inf_uF_{r_1}(u)
 +(1-\vartheta)\inf_uF_{r_2}(u)\\
 &=\vartheta\lambda_{r_1,t}+(1-\vartheta)\lambda_{r_2,t}.
\end{align*}
Therefore $r\mapsto\lambda_{r,t}$ is concave.

\textit{Step 4: the slope as $r\to+\infty$.}
There exists $C>0$ such that
\begin{equation}\label{eq:F1-lower}
 F_1(u)\ge-C
\end{equation}
for every $u\ne0$, because
\[
 \energy_{1,t}(u)
 \ge-\|\mu_1^-\|_\infty\int_Mu^2dV_g
 \ge-\|\mu_1^-\|_\infty\Vol(M,g)^{1/n}D(u).
\]
Using \eqref{eq:F-affine} with $s=1$ and $b(u)\le\Tmass$,
\[
 F_r(u)
 \ge-C-c_n(r-1)\Tmass.
\]
Taking infima and dividing by $r>0$ gives
\[
 \liminf_{r\to+\infty}\frac{\lambda_{r,t}}r
 \ge-c_n\Tmass.
\]

Conversely, fix $\eps>0$.  By \eqref{eq:holder-sup}, choose a smooth nonzero $u_\eps$ such that
\[
 b(u_\eps)>\Tmass-\eps.
\]
Then
\[
 \lambda_{r,t}
 \le F_r(u_\eps)
 =F_1(u_\eps)-c_n(r-1)b(u_\eps).
\]
Divide by $r$ and let $r\to+\infty$:
\[
 \limsup_{r\to+\infty}\frac{\lambda_{r,t}}r
 \le-c_n(\Tmass-\eps).
\]
Letting $\eps\downarrow0$ proves \eqref{eq:asymptotic-plus}.

\textit{Step 5: the slope as $r\to-\infty$.}
If $r\le1$, then \eqref{eq:Q-affine} gives
\[
 \energy_{r,t}(u)=\energy_{1,t}(u)+c_n(1-r)\int_MWu^2dV_g\ge \energy_{1,t}(u).
\]
Hence
\[
 \lambda_{1,t}\le\lambda_{r,t}\le t\Lambda_n
 \qquad(r\le1).
\]
Thus $\lambda_{r,t}$ remains bounded as $r\to-\infty$, and division by $r$ proves \eqref{eq:asymptotic-minus}.
\end{proof}

\begin{corollary}[A variational characterization of the LCK case]\label{cor:lck-slope}
Assume $n\ge3$.  Then for every $t>0$,
\begin{equation}\label{eq:lck-slope}
 \omega\text{ is LCK}
 \quad\Longleftrightarrow\quad
 \lim_{r\to+\infty}\frac{\lambda_{r,t}([\omega])}{r}=0.
\end{equation}
\end{corollary}

\begin{proof}
By Theorem \ref{thm:r-global}, the limit vanishes if and only if $\Tmass([\omega])=0$, equivalently $\tf\equiv0$.  The conclusion follows from Proposition \ref{prop:LCK-boundary}.
\end{proof}

\begin{remark}\label{rem:ae-differentiable}
Since a concave function is locally absolutely continuous and differentiable almost everywhere, $\lambda_{r,t}$ has a derivative for almost every $r$.  The theorem identifies its limiting slope at $+\infty$ without requiring differentiability at any finite parameter.
\end{remark}

\section{Critical thresholds and Attainment}\label{sec:saturation}

For fixed $t>0$, define the set
\begin{equation}\label{eq:saturation-set}
 \mathscr S_t([\omega])
 :=\{r\in\R:\lambda_{r,t}([\omega])=t\Lambda_n\}.
\end{equation}

\begin{theorem}\label{thm:saturation-halfline}
Assume $W=|\tf|^2\not\equiv0$ and $\mathscr S_t([\omega])\ne\varnothing$.  Then there exists a unique finite number
\[
 r_\sharp=r_\sharp(t,[\omega])
\]
such that
\begin{equation}\label{eq:saturation-halfline}
 \mathscr S_t([\omega])=(-\infty,r_\sharp].
\end{equation}
If $r_*\in\mathscr S_t([\omega])$ is arbitrary and
\begin{equation}\label{eq:defect-delta}
 \delta_*
 :=\inf_{\A(u)>0}
 \frac{\energy_{r_*,t}(u)-t\Lambda_nD(u)}{\A(u)},
 \qquad
 \A(u):=\int_MWu^2dV_g,
\end{equation}
then $\delta_*\in[0,\infty)$ and
\begin{equation}\label{eq:rsharp-formula}
 r_\sharp=r_*+\frac{\delta_*}{c_n}
 =r_*+\frac{2\delta_*}{n-1}.
\end{equation}
Moreover, for every $r<r_\sharp$, the critical infimum $\lambda_{r,t}=t\Lambda_n$ is not attained.
\end{theorem}

\begin{proof}
\textit{Step 1: downward closure of the equality set.}
Take $r_*\in\mathscr S_t([\omega])$.  Then
\begin{equation}\label{eq:defect-nonnegative}
 \energy_{r_*,t}(u)-t\Lambda_nD(u)\ge0
\end{equation}
for every $u\in H^1(M)$.  If $r<r_*$, formula \eqref{eq:Q-affine} gives
\[
 \energy_{r,t}(u)
 =\energy_{r_*,t}(u)+c_n(r_*-r)\A(u)
 \ge t\Lambda_nD(u).
\]
Thus $\lambda_{r,t}\ge t\Lambda_n$.  The universal bound \eqref{eq:sharp-upper} yields the reverse inequality, so
\[
 \lambda_{r,t}=t\Lambda_n.
\]
Therefore $\mathscr S_t([\omega])$ is downward closed.

\textit{Step 2: finiteness of the upper endpoint.}
Since $W\not\equiv0$, Theorem \ref{thm:r-global} gives
\[
 \frac{\lambda_{r,t}}r\longrightarrow-c_n\Tmass([\omega])<0
 \qquad(r\to+\infty).
\]
Hence $\lambda_{r,t}\to-\infty$ linearly as $r\to+\infty$.  In particular, equality with the positive number $t\Lambda_n$ cannot persist for arbitrarily large $r$.  Thus
\[
 r_\sharp:=\sup\mathscr S_t([\omega])<+\infty.
\]
Continuity of $r\mapsto\lambda_{r,t}$, supplied by the Lipschitz estimate \eqref{eq:Lipschitz}, gives
\[
 \lambda_{r_\sharp,t}=t\Lambda_n.
\]
Together with downward closure, this proves \eqref{eq:saturation-halfline}.

\textit{Step 3: the variational formula for $r_\sharp$.}
Because $r_*\in\mathscr S_t$, the numerator in \eqref{eq:defect-delta} is nonnegative; hence $\delta_*\ge0$.  Since $W\not\equiv0$, there is a smooth $u$ with $\A(u)>0$, so $\delta_*<+\infty$.

For arbitrary $r$,
\begin{equation}\label{eq:defect-shift}
 \energy_{r,t}(u)-t\Lambda_nD(u)
 =\energy_{r_*,t}(u)-t\Lambda_nD(u)-c_n(r-r_*)\A(u).
\end{equation}
If $\A(u)>0$, the definition of $\delta_*$ yields
\[
 \energy_{r_*,t}(u)-t\Lambda_nD(u)\ge\delta_*\A(u).
\]
If $\A(u)=0$, the left-hand side is nonnegative by \eqref{eq:defect-nonnegative}.
Therefore, whenever
\[
 r\le r_*+\frac{\delta_*}{c_n},
\]
formula \eqref{eq:defect-shift} implies
\[
 \energy_{r,t}(u)\ge t\Lambda_nD(u)
\]
for all $u$.  Hence such $r$ are saturated.

Conversely, suppose
\[
 r>r_*+\frac{\delta_*}{c_n}.
\]
Then $c_n(r-r_*)>\delta_*$.  By the definition of the infimum, there exists $u$ with $\A(u)>0$ such that
\[
 \frac{\energy_{r_*,t}(u)-t\Lambda_nD(u)}{\A(u)}
 <c_n(r-r_*).
\]
Equation \eqref{eq:defect-shift} gives
\[
 \energy_{r,t}(u)<t\Lambda_nD(u),
\]
so $\lambda_{r,t}<t\Lambda_n$.  This proves \eqref{eq:rsharp-formula}.

\textit{Step 4: nonattainment in the interior.}
Fix $r<r_\sharp$ and suppose for contradiction that the infimum is attained by some nonzero $u$.  Replacing $u$ by $|u|$ does not increase the gradient energy, and the Euler--Lagrange equation plus elliptic regularity and the strong maximum principle yield a smooth representative with
\[
 u>0.
\]
Because $W\ge0$ and $W\not\equiv0$,
\[
 \A(u)=\int_MWu^2dV_g>0.
\]
Since $r_\sharp$ is also saturated,
\begin{align*}
 \energy_{r_\sharp,t}(u)
 &=\energy_{r,t}(u)-c_n(r_\sharp-r)\A(u)\\
 &<\energy_{r,t}(u)\\
 &=t\Lambda_nD(u).
\end{align*}
This contradicts $\lambda_{r_\sharp,t}=t\Lambda_n$.  Hence no $r<r_\sharp$ admits a critical-level minimizer.
\end{proof}

\begin{corollary}[Weighted mass of interior minimizing sequences]\label{cor:interior-sequence}
Fix $r<r_\sharp$ and normalize a minimizing sequence by $D(u_j)=1$.  Then
\begin{equation}\label{eq:weighted-mass-zero}
 \int_MW u_j^2dV_g\longrightarrow0.
\end{equation}
\end{corollary}

\begin{proof}
Since $r_\sharp$ is saturated,
\[
 \energy_{r_\sharp,t}(u_j)\ge t\Lambda_nD(u_j)=t\Lambda_n.
\]
On the other hand,
\[
 \energy_{r,t}(u_j)
 =\energy_{r_\sharp,t}(u_j)+c_n(r_\sharp-r)\A(u_j).
\]
Thus
\[
 0\le c_n(r_\sharp-r)\A(u_j)
 \le \energy_{r,t}(u_j)-t\Lambda_n\longrightarrow0,
\]
which proves \eqref{eq:weighted-mass-zero}.
\end{proof}

Theorem \ref{thm:saturation-halfline} assumes that $\mathscr S_t([\omega])\ne\varnothing$. The following sufficient condition ensures this.

\begin{proposition}\label{prop:saturation-existence}
Assume
\begin{equation}\label{eq:W-positive}
 W=|\tf|^2\ge w_0>0
\end{equation}
on $M$.  Then $\mathscr S_t([\omega])\ne\varnothing$ for every $t>0$.  More precisely, if
\begin{equation}\label{eq:r-sufficient}
 r\le
 1-\frac{\bigl(t\Lambda_nB_0(g)-\min_M\mu^{1,t}\bigr)_+}
 {c_nw_0},
\end{equation}
then $\lambda_{r,t}=t\Lambda_n$.
\end{proposition}

\begin{proof}
Put
\[
 D_0:=t\Lambda_nB_0(g)-\min_M\mu^{1,t}.
\]
Condition \eqref{eq:r-sufficient} implies $r\le1$, so the coefficient $c_n(1-r)$ is nonnegative.  Since $W\ge w_0$, the exact reduction therefore gives
\begin{align*}
 \mu^{r,t}(x)
 &=\mu^{1,t}(x)+c_n(1-r)W(x)\\
 &\ge\min_M\mu^{1,t}+c_n(1-r)w_0\\
 &\ge\min_M\mu^{1,t}+(D_0)_+\\
 &\ge t\Lambda_nB_0(g).
\end{align*}
The positive part in \eqref{eq:r-sufficient} is essential when $D_0<0$: it keeps $r\le1$, so that replacing $W$ by its lower bound $w_0$ preserves the direction of the inequality.
Therefore, by \eqref{eq:at-K} and the optimal Sobolev inequality \eqref{eq:sharp-compact},
\begin{align*}
 \energy_{r,t}(u)
 &\ge a_t\int|\nabla u|^2+t\Lambda_nB_0(g)\int u^2\\
 &=t\Lambda_n\left(K_{2n}^2\int|\nabla u|^2+B_0(g)\int u^2\right)\\
 &\ge t\Lambda_nD(u).
\end{align*}
Hence $\lambda_{r,t}\ge t\Lambda_n$, and the universal upper bound forces equality.
\end{proof}

\subsection{The endpoint barrier and critical-level attainment}\label{sec:endpoint}
At the critical endpoint, by Theorem \ref{thm:pointwise-Aubin},
if $\lambda_{r,t}([\omega])=t\Lambda_n,$
then
\begin{equation}\label{eq:E-nonnegative}
 E_{r,t}\ge0
\end{equation}
pointwise on $M$. At points where $W>0$, the inequality can be solved explicitly for $r$.

\begin{definition}\label{def:Rloc}
Assume $W\not\equiv0$.  Define the local barrier
\begin{equation}\label{eq:Rloc}
 R_{\mathrm{loc}}(t;\omega)
 :=\inf_{\{x:W(x)>0\}}
 \left[
 1+\frac{2E_{1,t}(x)}{(2n-1)(n-1)W(x)}
 \right].
\end{equation}
\end{definition}

\begin{corollary}\label{cor:rsharp-local}
Under the hypotheses of Theorem \ref{thm:saturation-halfline},
\begin{equation}\label{eq:rsharp-local}
 r_\sharp\le R_{\mathrm{loc}}(t;\omega).
\end{equation}
\end{corollary}

\begin{proof}
At the endpoint, \eqref{eq:E-nonnegative} gives $E_{r_\sharp,t}\ge0$.  By \eqref{eq:E-reduction}, at every point with $W>0$,
\[
 E_{1,t}-(2n-1)c_n(r_\sharp-1)W\ge0.
\]
Solving for $r_\sharp$ and using $c_n=(n-1)/2$ gives the pointwise bound inside \eqref{eq:Rloc}.  Taking the infimum proves the claim.
\end{proof}

\begin{theorem}[Endpoint attainment]\label{thm:endpoint-attainment}
Assume the hypotheses of Theorem \ref{thm:saturation-halfline}.  If
\begin{equation}\label{eq:strict-endpoint}
 E_{r_\sharp,t}>0
 \qquad\text{everywhere on }M,
\end{equation}
then the critical infimum
\[
 \lambda_{r_\sharp,t}=t\Lambda_n
\]
is attained by a smooth positive function $u_\sharp$.  After normalizing $D(u_\sharp)=1$, it solves
\begin{equation}\label{eq:endpoint-pde}
 a_t\Delta u_\sharp+\mu^{r_\sharp,t}(\omega)u_\sharp
 =t\Lambda_nu_\sharp^{N-1}.
\end{equation}
Consequently $\widehat\omega=u_\sharp^{2/(n-1)}\omega$
has unit volume and constant deformed curvature $\mu^{r_\sharp,t}(\widehat\omega)=t\Lambda_n.$
\end{theorem}

\begin{proof}
Set again $h_\sharp:=\frac{\mu^{r_\sharp,t}}{a_t}.$ The endpoint equality is exactly
\[
 \lambda(h_\sharp)=K_{2n}^{-2},
\]
so $h_\sharp$ is weakly critical.  Since the endpoint invariant is positive, Lemma \ref{lem:coercivity} implies that $\Delta+h_\sharp$ is coercive.

For $\eps>0$, define
\[
 r_\eps:=r_\sharp+\eps,
 \qquad
 h_\eps:=\frac{\mu^{r_\eps,t}}{a_t}.
\]
The exact reduction gives
\begin{equation}\label{eq:h-eps}
 h_\eps
 =h_\sharp-\frac{c_n\eps}{a_t}W
 \le h_\sharp,
\end{equation}
with $h_\eps\not\equiv h_\sharp$ because $W\not\equiv0$.  Moreover $h_\eps\to h_\sharp$ in $C^\infty$ as $\eps\downarrow0$.  By the definition of $r_\sharp$,
\[
 \lambda_{r_\eps,t}<t\Lambda_n,
\]
so $h_\eps$ is subcritical.  The Lipschitz continuity of $\lambda_{r,t}$ implies
\[
 \lambda_{r_\eps,t}\longrightarrow t\Lambda_n>0,
\]
and therefore $\lambda_{r_\eps,t}>0$ for all sufficiently small $\eps$.  Another application of Lemma \ref{lem:coercivity} shows that $\Delta+h_\eps$ is coercive for such $\eps$.

It remains to check Collion's strict pointwise hypothesis.  In real dimension $m=2n$,
\[
 \frac{4(m-1)}{m-2}h_\sharp>\scal(g)
\]
is equivalent to
\[
 \frac{2n-1}{t}\mu^{r_\sharp,t}>\scal(g),
\]
which is precisely \eqref{eq:strict-endpoint}.  Thus every hypothesis of Theorem \ref{thm:collion-input}(ii) is satisfied by the family $h_\eps$.  Collion's theorem gives a positive minimizer for the endpoint quotient.

Normalize it by $D(u_\sharp)=1$.  The Euler--Lagrange equation is \eqref{eq:endpoint-pde}.  The conformal law \eqref{eq:conformal-u} then gives
\[
 \mu^{r_\sharp,t}(u_\sharp^{2/(n-1)}\omega)
 =t\Lambda_n,
\]
and $D(u_\sharp)=1$ is exactly the unit-volume condition because
\[
 dV_{u_\sharp^{2/(n-1)}\omega}=u_\sharp^N dV_\omega.
\]
\end{proof}

\begin{corollary}[Borderline alternative]\label{cor:borderline}
Under the hypotheses of Theorem \ref{thm:saturation-halfline}, if the endpoint infimum is not attained, then
\begin{equation}\label{eq:borderline-zero}
 \min_M E_{r_\sharp,t}=0.
\end{equation}
If in addition $W>0$ everywhere, then
\begin{equation}\label{eq:borderline-Rloc}
 r_\sharp=R_{\mathrm{loc}}(t;\omega).
\end{equation}
Equivalently, if $W>0$ and $r_\sharp<R_{\mathrm{loc}}(t;\omega)$, then the endpoint is attained.
\end{corollary}

\begin{proof}
$E_{r_\sharp,t}\ge0$.  If it were strictly positive everywhere, Theorem \ref{thm:endpoint-attainment} would give a minimizer.  Thus nonattainment forces a zero, proving \eqref{eq:borderline-zero}.

Assume now $W>0$ everywhere.  The ratio defining $R_{\mathrm{loc}}$ is continuous.  If $r_\sharp<R_{\mathrm{loc}}$, then the pointwise inequality obtained from \eqref{eq:E-reduction} is strict at every point, so $E_{r_\sharp,t}>0$ on the compact manifold.  The endpoint would then be attained. Therefore, nonattainment forces equality in \eqref{eq:borderline-Rloc}.  The final statement is the contrapositive.
\end{proof}

\subsection{The connection-parameter phase diagram}\label{sec:phase}

We first describe the parameter at which positivity of the quadratic
form is lost.  This does not require the saturation set
$\mathscr S_t([\omega])$ to be nonempty.

\begin{proposition}[The positivity endpoint]
\label{prop:positivity-endpoint}
Fix $t>0$ and suppose that $W\not\equiv0$.  Assume that
$\lambda_{s,t}([\omega])>0$ for some $s\in\R$, and set
\begin{equation}\label{eq:gamma-st}
 \gamma_{s,t}
 :=
 \inf_{\A(u)>0}
 \frac{\energy_{s,t}(u)}{\A(u)},
 \qquad
 \A(u):=\int_M Wu^2\,dV_g .
\end{equation}
Then
\begin{equation}\label{eq:gamma-lower}
 0<\frac{\lambda_{s,t}([\omega])}{\Tmass([\omega])}
 \leq \gamma_{s,t}<\infty .
\end{equation}
Moreover, the infimum in \eqref{eq:gamma-st} is attained, after
normalizing $\A(v)=1$, by a smooth positive function $v$ satisfying
\begin{equation}\label{eq:weighted-eigenfunction}
 a_t\Delta_\omega v+\mu^{s,t}(\omega)v
 =
 \gamma_{s,t}Wv .
\end{equation}

Define
\begin{equation}\label{eq:r0-def}
 r_0
 :=
 s+\frac{\gamma_{s,t}}{c_n}.
\end{equation}
Then $r_0$ is independent of the choice of $s$ for which
$\lambda_{s,t}([\omega])>0$, is invariant under conformal changes of
the background metric, and is the unique parameter at which
$\lambda_{r,t}$ changes sign:
\begin{equation}\label{eq:positivity-transition}
 \lambda_{r,t}([\omega])
 \begin{cases}
 >0,& r<r_0,\\
 =0,& r=r_0,\\
 <0,& r>r_0.
 \end{cases}
\end{equation}
In particular,
\begin{equation}\label{eq:r0-lower-general}
 r_0-s
 \geq
 \frac{\lambda_{s,t}([\omega])}
 {c_n\Tmass([\omega])}.
\end{equation}
\end{proposition}

\begin{proof}
Since $W\not\equiv0$, the admissible set in \eqref{eq:gamma-st} is
nonempty and $\Tmass([\omega])>0$.  By the definition of
$\lambda_{s,t}$ and Lemma \ref{lem:holder-sharp},
\[
 \energy_{s,t}(u)
 \geq
 \lambda_{s,t}D(u),
 \qquad
 \A(u)
 \leq
 \Tmass([\omega])D(u).
\]
Hence, whenever $\A(u)>0$,
\[
 \frac{\energy_{s,t}(u)}{\A(u)}
 \geq
 \frac{\lambda_{s,t}}{\Tmass([\omega])},
\]
which proves the lower bound in \eqref{eq:gamma-lower}.  Finiteness
follows by evaluating the quotient on any smooth function $u$ with
$\A(u)>0$.

Let $(u_j)$ be a minimizing sequence normalized by $\A(u_j)=1$.
Since $\lambda_{s,t}>0$, Lemma \ref{lem:coercivity} gives a uniform
$H^1$ bound.  Passing to a subsequence,
\[
 u_j\rightharpoonup v
 \quad\text{in }H^1(M),
 \qquad
 u_j\longrightarrow v
 \quad\text{in }L^2(M).
\]
As $W$ is smooth, $\A(u_j)\longrightarrow\A(v),$
and therefore $\A(v)=1$.  The Dirichlet term is weakly lower
semicontinuous, while the potential term converges by the strong
$L^2$ convergence.  Thus
\[
 \energy_{s,t}(v)
 \leq
 \liminf_{j\to\infty}\energy_{s,t}(u_j)
 =
 \gamma_{s,t},
\]
so $v$ attains the infimum.

Replacing $v$ by $|v|$ if necessary, we may assume $v\geq0$.
The Euler--Lagrange equation for the constraint $\A(v)=1$ is
\eqref{eq:weighted-eigenfunction}.  Elliptic regularity and the
strong maximum principle give $v\in C^\infty(M)$ and $v>0$.

The affine dependence \eqref{eq:Q-affine} gives, for every
$q\in\R$ and every $u$ with $\A(u)>0$,
\[
 \frac{\energy_{q,t}(u)}{\A(u)}
 =
 \frac{\energy_{s,t}(u)}{\A(u)}
 -
 c_n(q-s).
\]
Consequently, if $\gamma_{q,t}$ denotes the infimum
\eqref{eq:gamma-st} with $s$ replaced by $q$, then
\begin{equation}\label{eq:gamma-affine}
 \gamma_{q,t}
 =
 \gamma_{s,t}-c_n(q-s).
\end{equation}
In particular,
\[
 q+\frac{\gamma_{q,t}}{c_n}
 =
 s+\frac{\gamma_{s,t}}{c_n},
\]
so the value $r_0$ in \eqref{eq:r0-def} does not depend on the
base point $s$.

We next prove \eqref{eq:positivity-transition}.  If $r<s$, the
monotonicity of $r\mapsto\lambda_{r,t}$ gives $\lambda_{r,t}\geq\lambda_{s,t}>0.$ Suppose instead that $s\leq r<r_0$.  Put $\delta:=c_n(r-s).$
Then $0\leq\delta<\gamma_{s,t}$.  By the definition of
$\gamma_{s,t}$,
\[
 \gamma_{s,t}\A(u)\leq\energy_{s,t}(u),
\]
and hence
\begin{align*}
 \energy_{r,t}(u)
 &=
 \energy_{s,t}(u)-\delta\A(u)\\
 &\geq
 \left(1-\frac{\delta}{\gamma_{s,t}}\right)
 \energy_{s,t}(u)\\
 &\geq
 \left(1-\frac{\delta}{\gamma_{s,t}}\right)
 \lambda_{s,t}D(u).
\end{align*}
Thus $\lambda_{r,t}>0$.

At $r=r_0$, the definition of $r_0$ gives $\energy_{r_0,t}(u) = \energy_{s,t}(u)-\gamma_{s,t}\A(u) \geq0 $ for every $u$, while the minimizing function $v$ satisfies
\[
 \energy_{r_0,t}(v)
 =
 \energy_{s,t}(v)-\gamma_{s,t}\A(v)
 =
 0.
\]
Therefore $\lambda_{r_0,t}=0$.  Finally, if $r>r_0$, then
\[
 \energy_{r,t}(v)
 =
 \energy_{r_0,t}(v)
 -
 c_n(r-r_0)\A(v)
 <0,
\]
so $\lambda_{r,t}<0$.  This proves \eqref{eq:positivity-transition}
and also the uniqueness of the zero.

It remains to check conformal invariance.  Let $\widehat\omega
 = \varphi^{2/(n-1)}\omega,$ $\varphi>0.$ The conformal covariance of the quadratic form gives $ \widehat{\energy}_{s,t}(u) = \energy_{s,t}(\varphi u).$
Moreover, by the conformal transformation of $W$,
\[
 \int_M\widehat W\,u^2\,dV_{\widehat g}
 =
 \int_M W(\varphi u)^2\,dV_g.
\]
Thus the quotient defining $\gamma_{s,t}$ is unchanged under the
bijection $u\mapsto\varphi u$.  Hence $\gamma_{s,t}$, and therefore
$r_0$, depends only on the Hermitian conformal class.

Finally, \eqref{eq:r0-lower-general} follows immediately from
\eqref{eq:gamma-lower} and \eqref{eq:r0-def}.
\end{proof}

We now assume that
$\mathscr S_t([\omega])\neq\varnothing$, and $r_\sharp$ denotes the
endpoint given by Theorem \ref{thm:saturation-halfline}.

\begin{theorem}[Connection-parameter phase diagram]
\label{thm:phase}
Assume $W\not\equiv0$ and
$\mathscr S_t([\omega])\neq\varnothing$.  Let $r_\sharp$ be the
endpoint of the interval.  Then the positivity endpoint
$r_0$ of Proposition \ref{prop:positivity-endpoint} satisfies
\begin{equation}\label{eq:r0-saturated}
 r_0
 =
 r_\sharp
 +
 \frac{1}{c_n}
 \inf_{\A(u)>0}
 \frac{\energy_{r_\sharp,t}(u)}{\A(u)}
 >
 r_\sharp .
\end{equation}
Moreover,
\begin{equation}\label{eq:phase-gap}
 r_0-r_\sharp
 \geq
 \frac{2t\Lambda_n}
 {(n-1)\Tmass([\omega])}.
\end{equation}
The invariant satisfies
\begin{equation}\label{eq:phase-diagram}
 \lambda_{r,t}([\omega])
 =
 \begin{cases}
 t\Lambda_n,
     & r\leq r_\sharp,\\[2mm]
 \text{a number in }(0,t\Lambda_n),
     & r_\sharp<r<r_0,\\[2mm]
 0,
     & r=r_0,\\[2mm]
 <0,
     & r>r_0.
 \end{cases}
\end{equation}
Furthermore, the map
\[
 r\longmapsto\lambda_{r,t}([\omega])
\]
is strictly decreasing on $(r_\sharp,\infty)$.

For $r<r_\sharp$ the critical level is not attained.  For
$r>r_\sharp$ the infimum is attained by
Theorem  \ref{thm:sharp-threshold}; at $r=r_\sharp$ the endpoint attainment
criterion of Theorem \ref{thm:endpoint-attainment} applies.
\end{theorem}

\begin{proof}
Since $\lambda_{r_\sharp,t}=t\Lambda_n>0,$
Proposition \ref{prop:positivity-endpoint} applies with $s=r_\sharp$.
This gives \eqref{eq:r0-saturated} and the sign change at $r_0$.
It also gives
\[
 r_0-r_\sharp
 \geq
 \frac{\lambda_{r_\sharp,t}}
 {c_n\Tmass([\omega])}
 =
 \frac{t\Lambda_n}
 {c_n\Tmass([\omega])}
 =
 \frac{2t\Lambda_n}
 {(n-1)\Tmass([\omega])},
\]
which proves \eqref{eq:phase-gap}. By Theorem \ref{thm:saturation-halfline}, $\lambda_{r,t}=t\Lambda_n$ for $r\leq r_\sharp.$

If $r_\sharp<r<r_0$, then
Proposition \ref{prop:positivity-endpoint} gives $\lambda_{r,t}>0$, while
$r\notin\mathscr S_t([\omega])$ gives
$\lambda_{r,t}<t\Lambda_n$.  The remaining two cases in
\eqref{eq:phase-diagram} follow directly from
Proposition \ref{prop:positivity-endpoint}.

It remains to prove strict monotonicity.  Let $r_\sharp<s<r.$
Since $s>r_\sharp$, $\lambda_{s,t}<t\Lambda_n.$ By Theorem \ref{thm:sharp-threshold}, the infimum is attained by a smooth
positive minimizer $u_s$.  Normalize $D(u_s)=1$.  Since
$W\not\equiv0$ and $u_s>0$, $\A(u_s)>0.$
Using \eqref{eq:Q-affine},
\begin{align*}
 \lambda_{r,t}
 &\leq
 \energy_{r,t}(u_s)\\
 &=
 \energy_{s,t}(u_s)-c_n(r-s)\A(u_s)\\
 &=
 \lambda_{s,t}-c_n(r-s)\A(u_s)\\
 &<
 \lambda_{s,t}.
\end{align*}
Thus $r\mapsto\lambda_{r,t}$ is strictly decreasing on
$(r_\sharp,\infty)$.

The statements concerning attainment follow from
Theorem \ref{thm:saturation-halfline}, \ref{thm:sharp-threshold} and \ref{thm:endpoint-attainment}.
\end{proof}

\begin{remark}\label{rem:two-parameters}
The parameters $r_\sharp$ and $r_0$ have different
meanings.  The first is the endpoint of the saturated interval
\[
 \{r:\lambda_{r,t}=t\Lambda_n\},
\]
whereas $r_0$ is the unique zero of $\lambda_{r,t}$.  The parameter
$r_0$ is defined whenever $\lambda_{s,t}>0$ for some $s$, even if
$\mathscr S_t([\omega])$ is empty.  When the saturated interval is
nonempty, \eqref{eq:phase-gap} gives a quantitative separation
between the two parameters.
\end{remark}

\subsection{Constant potentials and exact critical levels}
\label{sec:constant}

When $W$ and $\mu^{1,t}$ are constant, the optimal second Sobolev
constant $B_0(g)$ gives explicit formulas for both $r_\sharp$ and
$r_0$.  We use the Djadli--Druet extremal criterion
\eqref{eq:DD}.

\begin{theorem}
\label{thm:constant-r}
Suppose
\[
 W\equiv w>0,
 \qquad
 \mu^{1,t}\equiv C_t .
\]
Then
\begin{equation}\label{eq:constant-rsharp}
 r_\sharp
 =
 1+\frac{C_t-t\Lambda_nB_0(g)}{c_nw},
 \qquad
 r_0
 =
 1+\frac{C_t}{c_nw}.
\end{equation}
More precisely,
\[
 \lambda_{r,t}=t\Lambda_n
 \quad\Longleftrightarrow\quad
 r\leq r_\sharp,
\]
and the critical level is not attained for $r<r_\sharp$.  At
$r=r_\sharp$, it is attained if and only if the optimal Sobolev
inequality \eqref{eq:sharp-compact} has a nonzero extremal.
Consequently, the endpoint is attained under \eqref{eq:DD}, and in
particular whenever $\scal(g)\leq0$.

A constant function is an endpoint minimizer if and only if
\[
 B_0(g)=\Vol(M,g)^{-1/n}.
\]
If
\[
 B_0(g)>\Vol(M,g)^{-1/n},
\]
then every endpoint minimizer is nonconstant.
\end{theorem}

\begin{proof}
Set $P_r:=C_t-c_n(r-1)w.$ Then $\mu^{r,t}\equiv P_r$, and
\begin{equation}\label{eq:constant-defect}
\begin{aligned}
 \energy_{r,t}(u)-t\Lambda_nD(u)
 &=
 t\Lambda_n
 \left(
 K_{2n}^2\int_M|\nabla u|^2\,dV_g
 +
 B_0(g)\int_Mu^2\,dV_g
 -
 D(u)
 \right)\\
 &\quad+
 \bigl(P_r-t\Lambda_nB_0(g)\bigr)
 \int_Mu^2\,dV_g .
\end{aligned}
\end{equation}
By the optimal Sobolev inequality, the first term on the
right-hand side is nonnegative.  Hence, if
\[
 P_r\geq t\Lambda_nB_0(g),
\]
then
\[
 \energy_{r,t}(u)\geq t\Lambda_nD(u)
\]
for every $u$, and therefore
$\lambda_{r,t}=t\Lambda_n$ by the universal upper bound.

Conversely, suppose $P_r<t\Lambda_nB_0(g).$
Then
\[
 B:=\frac{P_r}{t\Lambda_n}<B_0(g).
\]
By the optimality of $B_0(g)$, there exists a nonzero $u$ such that
\[
 D(u)
 >
 K_{2n}^2\int_M|\nabla u|^2\,dV_g
 +
 B\int_Mu^2\,dV_g.
\]
Equivalently,
\[
 \energy_{r,t}(u)<t\Lambda_nD(u),
\]
so $\lambda_{r,t}<t\Lambda_n$.  Thus the spherical level occurs
exactly when
\[
 P_r\geq t\Lambda_nB_0(g),
\]
which is equivalent to $r\leq r_\sharp$ with $r_\sharp$ as in
\eqref{eq:constant-rsharp}.

At $r=r_\sharp$, the second term in
\eqref{eq:constant-defect} vanishes.  Therefore equality
\[
 \energy_{r_\sharp,t}(u)=t\Lambda_nD(u)
\]
holds precisely for nonzero extremals of
\eqref{eq:sharp-compact}.  If $r<r_\sharp$, then
\[
 P_r-t\Lambda_nB_0(g)>0,
\]
so equality is impossible for a nonzero $u$.  This proves both the
attainment statement and nonattainment in the interior of the
saturated interval.  The Djadli--Druet conclusion follows from
\eqref{eq:DD}.

It remains to determine $r_0$.  If $P_r>0$, then, using
\eqref{eq:sharp-compact},
\[
 \energy_{r,t}(u)
 \geq
 \min\left\{
 \frac{a_t}{K_{2n}^2},
 \frac{P_r}{B_0(g)}
 \right\}D(u),
\]
so $\lambda_{r,t}>0$.  If $P_r=0$, then
$\energy_{r,t}(u)\geq0$ and constants give
$\lambda_{r,t}=0$.  If $P_r<0$, testing with a nonzero constant
gives $\lambda_{r,t}<0$.  Hence the unique zero is characterized by
$P_{r_0}=0$, which gives the second formula in
\eqref{eq:constant-rsharp}.

Finally, at $r=r_\sharp$ the endpoint minimizers are precisely the
extremals of \eqref{eq:sharp-compact}.  A nonzero constant is such
an extremal exactly when
\[
 B_0(g)=\Vol(M,g)^{-1/n}.
\]
The last assertion follows.
\end{proof}

\begin{lemma}[A spectral test for nonconstant extremals]
\label{lem:spectral}
Let $V=\Vol(M,g)$.  Suppose there exists a smooth nonzero function
$\psi$ with
\[
 \int_M\psi\,dV_g=0
\]
and
\begin{equation}\label{eq:spectral-test}
 \frac{\displaystyle\int_M|\nabla\psi|^2\,dV_g}
      {\displaystyle\int_M\psi^2\,dV_g}
 <
 \Lambda_nV^{-1/n}.
\end{equation}
Then
\[
 B_0(g)>V^{-1/n}.
\]
\end{lemma}

\begin{proof}
For
\[
 u_\varepsilon:=1+\varepsilon\psi,
\]
with $|\varepsilon|$ sufficiently small, the zero-mean assumption
gives
\[
 \int_Mu_\varepsilon^N\,dV_g
 =
 V+
 \frac{N(N-1)}{2}\varepsilon^2
 \int_M\psi^2\,dV_g
 +
 O(\varepsilon^3).
\]
Therefore
\begin{equation}\label{eq:D-expansion-spectral}
 D(u_\varepsilon)
 =
 V^{2/N}
 +
 (N-1)V^{-1/n}\varepsilon^2
 \int_M\psi^2\,dV_g
 +
 O(\varepsilon^3).
\end{equation}
On the other hand,
\[
 \int_M|\nabla u_\varepsilon|^2\,dV_g
 =
 \varepsilon^2
 \int_M|\nabla\psi|^2\,dV_g,
\]
and
\[
 \int_Mu_\varepsilon^2\,dV_g
 =
 V+
 \varepsilon^2\int_M\psi^2\,dV_g.
\]
Hence the defect in \eqref{eq:sharp-compact} with
$B=V^{-1/n}$ is
\[
 \varepsilon^2
 \left(
 K_{2n}^2\int_M|\nabla\psi|^2\,dV_g
 -
 (N-2)V^{-1/n}\int_M\psi^2\,dV_g
 \right)
 +
 O(\varepsilon^3).
\]
Since
\[
 \frac{N-2}{K_{2n}^2}=\Lambda_n,
\]
condition \eqref{eq:spectral-test} makes this quantity negative for
all sufficiently small nonzero $\varepsilon$.  Thus
$V^{-1/n}$ is not admissible as the second constant in
\eqref{eq:sharp-compact}, and consequently
\[
 B_0(g)>V^{-1/n}.
\]
\end{proof}

\subsection{The critical value in the Chern deformation}

\begin{theorem}[An attained threshold in $t$]\label{thm:chern-threshold}
Assume that, on a fixed Hermitian background,
\begin{equation}\label{eq:affine-t}
 \mu^{1,t}=A+bt\quad\text{for all }t,\qquad A>0,\quad b\geq0,
\end{equation}
where $A,b$ are constants. Then
\begin{equation}\label{eq:tstar}
 t_*:=\frac{A}{\Lambda_nB_0(g)-b}\in(0,\ty).
\end{equation}
For $0<t\leq t_*$ one has $\lambda_{1,t}=t\Lambda_n$, whereas for $t>t_*$ one has $\lambda_{1,t}<t\Lambda_n$. The infimum is not attained for $t<t_*$ and is attained for every $t\geq t_*$. At $t_*$, its minimizers are exactly the nonzero extremals of \eqref{eq:sharp-compact}, up to taking the positive representative. All are nonconstant if $B_0(g)>V^{-1/n}$; constants are minimizers if $B_0(g)=V^{-1/n}$.

The map $t\mapsto\lambda_{1,t}$ is positive, nondecreasing, concave, and continuous on $(0,\infty)$. The ratio $\lambda_{1,t}/t$ is constant on $(0,t_*]$ and strictly decreasing on $[t_*,\infty)$. The number $t_*$ is unchanged by a constant rescaling of $g$.
\end{theorem}
\begin{proof}
At $t=\ty$, \eqref{eq:affine-t} gives the positive constant scalar curvature $\scal(g)=A+b\ty$. A compact Hermitian manifold of complex dimension $n\geq2$ cannot be conformally equivalent to the round sphere. Indeed, the only even spheres admitting almost complex structures are $\mathbb S^2$ and $\mathbb S^6$, and the round six-sphere admits no orthogonal integrable complex structure; see \cite{LeBrun1987,Ferreira2018}. The strict form of the Yamabe theorem therefore gives $\lambda_{1,\ty}<\ty\Lambda_n$. If $B_0(g)\leq \scal(g)/(\ty\Lambda_n)$, the sharp Sobolev inequality would instead give $\energy_{1,\ty}\geq\ty\Lambda_nD$. Consequently
\begin{equation}\label{eq:B0-strict-chern}
 \Lambda_nB_0(g)>\frac{A}{\ty}+b.
\end{equation}
This proves \eqref{eq:tstar}, and also the strict hypothesis \eqref{eq:DD}. Thus the second Sobolev inequality has an extremal.

For every $t>0$,
\begin{align}\label{eq:t-defect}
 \energy_{1,t}(u)-t\Lambda_nD(u)
 &=t\Lambda_n\left(K_{2n}^2\int_M\abs{\nabla u}^2\frac{\omega^n}{n!}+B_0(g)\int_M u^2\frac{\omega^n}{n!}-D(u)\right)\notag\\
 &\quad+\left(A-t(\Lambda_nB_0(g)-b)\right)\int_M u^2\frac{\omega^n}{n!}.
\end{align}
The assertions about the levels and their attainment now follow exactly as in Theorem~\ref{thm:constant-r}, using optimality of $B_0(g)$ and Theorem~\ref{thm:sharp-threshold} for the strict range.

For each nonzero $u$, the quotient is an affine nondecreasing function of $t$, since its coefficient is $[2\norm{\nabla u}_2^2/(n-1)+b\norm u_2^2]/D(u)$. Its infimum is therefore nondecreasing and concave, hence continuous on $(0,\infty)$. Positivity follows from coercivity of $a_t\Delta+A+bt$ and Sobolev embedding. If $t_2>t_1\geq t_*$, testing at a minimizer $u$ for $t_1$ gives
\[
 \frac{\lambda_{1,t_2}}{t_2}
 \leq\frac{\lambda_{1,t_1}}{t_1}
       -A\left(\frac1{t_1}-\frac1{t_2}\right)\frac{\norm u_2^2}{D(u)}
 <\frac{\lambda_{1,t_1}}{t_1}.
\]
Finally, under $g\mapsto cg$, the constants $A,b,B_0$ all scale by $c^{-1}$, proving invariance of $t_*$.
\end{proof}

\begin{proposition}[Minimizing sequences below $t_*$]\label{prop:t-concentration}
Under the hypotheses of Theorem~\ref{thm:chern-threshold}, fix $0<t<t_*$. If $D(u_j)=1$ and $\energy_{1,t}(u_j)\to t\Lambda_n$, then
\[
 \norm{u_j}_2\to0,\qquad
 \int_M\abs{\nabla u_j}^2\frac{\omega^n}{n!}\to K_{2n}^{-2}.
\]
The sequence converges weakly to zero in $H^1$ and strongly to zero in $L^p$ for $1\leq p<N$, but not strongly in $L^N$.
\end{proposition}
\begin{proof}
The coefficient of $\int u^2$ in \eqref{eq:t-defect} is strictly positive, and the Sobolev defect is nonnegative. Thus $\norm{u_j}_2\to0$. The energy identity gives the gradient limit. Boundedness in $H^1$, compact Sobolev embeddings, and the fixed $L^N$ normalization give the remaining conclusions.
\end{proof}

\section{Examples}\label{sec:examples}

\subsection{Iwasawa products and their finite covers}
Let $G=\C^3$ with complex Heisenberg multiplication
\[
 (z_1,z_2,z_3)(w_1,w_2,w_3)
 =(z_1+w_1,z_2+w_2,z_3+w_3+z_1w_2).
\]
Let $\Gamma$ be the subgroup with Gaussian-integer coordinates and $X=\Gamma\backslash G$. The left-invariant holomorphic coframe
\[
 \varphi^1=\dd z_1,\quad \varphi^2=\dd z_2,\quad
 \varphi^3=\dd z_3-z_1\dd z_2
\]
satisfies $\dd\varphi^1=\dd\varphi^2=0$ and $\dd\varphi^3=-\varphi^1\wedge\varphi^2$. Equip $X$ with
\[
 \omega_X=\frac{\sqrt{-1}}2\sum_{j=1}^3\varphi^j\wedge\bar\varphi^j.
\]
For $n\geq3$, take its product with a flat complex torus of dimension $n-3$; for $n=3$ this factor is a point. The product metric is balanced and non-K\"ahler. Direct computations give
\begin{equation}\label{eq:iwasawa-invariants}
 \scal(g)=-2,\qquad \scal^\Ch=0,\qquad W=\abs T^2=8,
 \qquad \mu^{r,t}=2(t-2n)-4(n-1)(r-1).
\end{equation}

\begin{corollary}[An attained non-Chern endpoint]\label{cor:iwasawa-endpoint}
On these Iwasawa products, for every $t>0$ $\mathscr S_t([\omega])$ has endpoint
\begin{equation}\label{eq:iwasawa-rsharp}
 r_\sharp(t)=1+\frac{2(t-2n)-t\Lambda_nB_0(g)}{4(n-1)},
\end{equation}
and its endpoint is attained. In particular, at $t=\ty$,
\[
 r_\sharp(\ty)=1-\frac{2+\ty\Lambda_nB_0(g)}{4(n-1)}<1.
\]
Thus a non-locally-conformally-K\"ahler class attains the spherical level for a connection different from the Chern connection.
\end{corollary}
\begin{proof}
Apply Theorem~\ref{thm:constant-r}. The Djadli--Druet strict inequality holds since $\scal(g)=-2$ and $B_0(g)>0$.
\end{proof}

For a positive integer $L$, let $\Gamma_L\subset\Gamma$ be the subgroup with $z_1\in L\Z[\sqrt{-1}]$. It has index $L^2$. Let $M_L$ be the corresponding Iwasawa cover, with the same fixed flat torus factor and the lifted metric $g_L$. Write $V_L=\Vol(M_L,g_L)=L^2V_1$.

\begin{corollary}[Nonconstant critical minimizers on covers]\label{cor:covers}
For all sufficiently large $L$, every minimizer at $r_{\sharp, L}(t)$ on $(M_L,g_L)$ is nonconstant, for every $t>0$.
\end{corollary}
\begin{proof}
The function $\psi_L=\cos(2\pi\operatorname{Re}z_1/L)$ is globally defined and satisfies
\[
 \Delta_{g_L}\psi_L=\frac{4\pi^2}{L^2}\psi_L.
\]
Its integral is zero by integration of this eigenvalue equation. Since
\[
 \frac{4\pi^2}{L^2}<\Lambda_n(L^2V_1)^{-1/n}
\]
for sufficiently large $L$, Lemma~\ref{lem:spectral} gives $B_0(g_L)>V_L^{-1/n}$. The endpoint is attained by Corollary~\ref{cor:iwasawa-endpoint}, and Theorem~\ref{thm:constant-r} rules out constant minimizers.
\end{proof}

\begin{remark} A lower bound on $B_0(g_L)$ yields an upper bound on $r_{\sharp, L}(t)$. These are some explicit numerical bounds for the $L=2$ Iwasawa cover for some values of $t$:
\begin{enumerate}
    \item at $t = 5$,
    \[r_{\sharp, 2}(5) = 1-\frac{2+5\Lambda_3B_0(g_2)}{8} \quad \text{and} \quad r_{\sharp, 2}(5) \leq -6.8993. \]
 \item at $t = 1$,
    \[r_{\sharp, 2}(1) = 1-\frac{10+\Lambda_3B_0(g_2)}{8} \quad \text{and} \quad r_{\sharp, 2}(1) \leq -1.7799. \]
\end{enumerate}
\end{remark}

\begin{remark}
The original invariant metric in \eqref{eq:iwasawa-invariants} has constant deformed curvature for every parameter. This alone says nothing about whether it minimizes the normalized total curvature. Corollary~\ref{cor:covers} shows that at the critical endpoint all minimizing metrics on the indicated covers require a nonconstant conformal factor.
\end{remark}

\subsection{Fubini--Study metrics}
On $\CP^n$, $n\geq2$, use
\[
 \omega_{\mathrm{FS}}=\sqrt{-1}\,\partial\bar\partial\log(1+\abs z^2).
\]
Then $W=0$, $\scal(g)=2n(n+1)$, and $V=(2\pi)^n/n!$. Consequently Theorem~\ref{thm:chern-threshold} applies with $A=2n(n+1)$ and $b=0$. All connections have the same threshold
\begin{equation}\label{eq:FS-threshold}
 t_{*,\mathrm{FS}}=\frac{2n(n+1)}{\Lambda_nB_0(g_{\mathrm{FS}})}\in(0,\ty).
\end{equation}
The spherical level is attained at this value and is unattained for smaller positive $t$.

For $t\geq\ty$, the invariant can be evaluated independently of $B_0$:
\begin{equation}\label{eq:FS-value}
 \lambda_{r,t}=C_n:=\frac{4\pi n(n+1)}{(n!)^{1/n}}.
\end{equation}
Indeed, the Fubini--Study metric is Einstein and is a Yamabe minimizer by the Yamabe theorem and Obata's rigidity theorem \cite{Obata1971}. Thus the quotient at $\ty$ is bounded below by $\scal(g)V^{1/n}=C_n$. Increasing $t$ only increases its gradient term, while constants still give $C_n$. No formula $\lambda_{r,t}=\min\{t\Lambda_n,C_n\}$ is asserted in the intermediate range; such a formula would require additional information about $B_0(g_{\mathrm{FS}})$ and its extremals.

\subsection{Nonconstant critical minimizers on Hopf manifolds}
Let
\[
 H^n_\ell=(\C^n\setminus\{0\})/\langle z\mapsto e^{-\ell}z\rangle,
 \qquad \ell>0,
\]
and put $s=\abs z^2$. For $a>0$, the form
\begin{equation}\label{eq:Hopf-metric}
 \omega_a=\frac{\sqrt{-1}}{a s^a}\partial\bar\partial s^a
 =\sqrt{-1}\left(\frac{\delta_{ij}}s+
              (a-1)\frac{\bar z_i z_j}{s^2}\right)\dd z^i\wedge\dd\bar z^j
\end{equation}
is positive and descends to $H^n_\ell$. It is locally conformally K\"ahler and has parallel Lee form. Direct calculations give
\begin{align}
 \scal(g_a)&=(n-1)(2n-a),\qquad
 \Vol(H^n_\ell,g_a)=\frac{2^{n+1}a\ell\pi^n}{(n-1)!},\label{eq:Hopf-data}\\
 \mu^{r,t}(\omega_a)&=(n-1)\bigl(2n(1-a)+at\bigr).\label{eq:Hopf-mu}
\end{align}
For $0<a<1$, Theorem~\ref{thm:chern-threshold} gives an attained threshold, independent of $r$,
\begin{equation}\label{eq:Hopf-tstar}
 t_*(a,\ell)=\frac{2n(n-1)(1-a)}{\Lambda_nB_0(g_a)-a(n-1)}.
\end{equation}

\begin{theorem}[Hopf example with nonconstant minimizers]\label{thm:Hopf}
For each $n\geq2$, set
\begin{equation}\label{eq:Hopf-period}
 a=\frac12,\qquad
 \ell_n=\frac{(n-1)!\Vol(\mathbb S^{2n})}{2\pi^n}
        =\frac{2^{2n}}{n\binom{2n}{n}}.
\end{equation}
On the resulting Hopf manifold, for every $r$ the spherical level is attained at
\begin{equation}\label{eq:Hopf-tau}
t_{*,n} =\frac{n(n-1)}{\Lambda_nB_0(g_{1/2})-(n-1)/2},
\end{equation}
and every minimizing conformal factor at $t_{*,n}$ is nonconstant. Moreover,
\begin{equation}\label{eq:Hopf-tau-bound}
 0<t_{*,n}<\frac{n(n-1)}{n2^{1/n}-(n-1)/2}<2n-1.
\end{equation}
The underlying complex manifold admits no K\"ahler metric.
\end{theorem}
\begin{proof}
The chosen volume is $V_n=2^{n-1}\Vol(\mathbb S^{2n})$. The globally defined function $\psi(z)=\operatorname{Re}(z_1/\abs z)$ satisfies $\Delta_{g_{1/2}}\psi=n\psi$. Since
\[
 n<\Lambda_nV_n^{-1/n}=n2^{1/n},
\]
Lemma~\ref{lem:spectral} gives $B_0(g_{1/2})>V_n^{-1/n}$. Theorem~\ref{thm:chern-threshold} proves attainment and nonconstancy. This strict lower bound for $B_0$ yields the middle upper bound in \eqref{eq:Hopf-tau-bound}. Its comparison with $2n-1$ follows, for example, from $2^{1/n}>1$ and
\[
 n(n-1)<\frac{(2n-1)(n+1)}2.
\]
Finally, $H^n_{\ell_n}$ is diffeomorphic to $\mathbb S^1\times\mathbb S^{2n-1}$, whose first Betti number is one, whereas a compact K\"ahler manifold has an even first Betti number.
\end{proof}

\begin{remark} These are some explicit numerical bounds for $t_{*,n}$:
\begin{enumerate}
    \item For $n = 2,$ $\ell_2 = \frac{4}{3}.$ And so, on $H^2_{\frac{4}{3}}$,
    \[t_{*,2} =\frac{4}{2\Lambda_nB_0(g_{1/2})-1} \quad \text{and} \quad 0 < t_{*,2} \leq 0.8351. \]

    \item For $n = 3,$ $\ell_2 = \frac{16}{15}.$ And so, on $H^3_{\frac{16}{15}}$,
    \[t_{*,3} =\frac{6}{\Lambda_nB_0(g_{1/2})-1} \quad \text{and} \quad 0 < t_{*,3} \leq 2.1391.\]
\end{enumerate}
\end{remark}

%These examples establish equality and attainment at the Chern spherical threshold in every complex dimension at least two, and show that the minimizing conformal factor need not be constant. They also illustrate the two different roles of Hermitian torsion: the Lee component influences the deformation in $t$, whereas the trace-free component is required for a nontrivial dependence on $r$.

\section{Prescribed deformed scalar curvature}\label{sec:prescribed}
We now fix $r\in\R$ and $t>0$ and consider
\begin{equation}\label{eq:prescribed-PDE}
L_{r,t}u
:=a_t\Delta_\omega u+\mu^{r,t}(\omega)u
=f u^{N-1},
\qquad u>0.
\end{equation}
By Proposition \ref{prop:conformal-law}, a positive solution is equivalent to a conformal metric with prescribed $\mu^{r,t}$-curvature $f$.

When $\lambda_{r,t}([\omega])\leq 0$, Theorem \ref{thm:sharp-threshold} supplies a conformal representative with non-positive constant $\mu^{r,t}$. We restrict to $\lambda_{r,t}([\omega]) > 0$, where a positive constant $\mu^{r,t}$ representative is not assumed. We therefore assume throughout this section that
\begin{equation}\label{eq:positive-class}
\lambda_{r,t}([\omega])>0.
\end{equation}

\subsection{The weighted variational problem}

A first necessary condition follows by testing \eqref{eq:prescribed-PDE} against $u$.

\begin{lemma}\label{lem:necessary-positive}
Under \eqref{eq:positive-class}, if \eqref{eq:prescribed-PDE} has a positive solution, then $\sup_M f>0$.
\end{lemma}

\begin{proof}
If $u>0$ solves \eqref{eq:prescribed-PDE}, then
\[
\energy_{r,t}(u)=\int_M f u^N\,\mathrm dV_\omega.
\]
The left-hand side is at least $\lambda_{r,t}\norm{u}_{L^N}^2>0$, so $f$ must be positive somewhere.
\end{proof}

Assume henceforth that $\sup_M f>0$. Define
\begin{equation}\label{eq:Af}
\mathcal A_f
:=
\left\{
v\in H^1(M):
v\ge0,
\quad
\int_M f v^N\,\mathrm dV_\omega=1
\right\}
\end{equation}
and
\begin{equation}\label{eq:lambda-f}
\lambda_f^{r,t}
:=\inf_{v\in\mathcal A_f}\energy_{r,t}(v).
\end{equation}
The set $\mathcal A_f$ is nonempty by choosing a nonzero test function supported in $\{f>0\}$. In particular,
\begin{equation}\label{eq:lambda-f-lower}
\lambda_f^{r,t}
\ge
\frac{\lambda_{r,t}}{(\sup_M f)^{2/N}}>0,
\end{equation}
because $1=\int f v^N\le(\sup f)\norm{v}_{L^N}^N$ for $v\in\mathcal A_f$.

\begin{theorem}[Weighted threshold]\label{thm:weighted-threshold}
Assume \eqref{eq:positive-class} and let $f\in C^\infty(M)$ satisfy $\sup_M f>0$. If
\begin{equation}\label{eq:weighted-condition}
\lambda_f^{r,t}
<
\frac{t\Lambda_n}{(\sup_M f)^{2/N}},
\end{equation}
then there exists a smooth positive solution of \eqref{eq:prescribed-PDE}.
\end{theorem}

\begin{proof}
For $2<q<N$, minimize $\energy_{r,t}$ under the subcritical constraint $\int_Mfv^q\,\mathrm dV_\omega=1$. Coercivity Lemma \ref{lem:coercivity} and the compact embedding $H^1\hookrightarrow L^q$ give a smooth positive minimizer $v_q$ satisfying
\[
L_{r,t}v_q=\lambda_{f,q}^{r,t} f v_q^{q-1}.
\]
A fixed test function supported in $\{f>0\}$ gives an upper bound for $\lambda_{f,q}^{r,t}$, hence Lemma \ref{lem:coercivity} gives a uniform $H^1$ bound. The standard comparison of the $q$- and $N$-constraints shows $\lambda_{f,q}^{r,t}\to\lambda_f^{r,t}$ as $q\uparrow N$.

After dividing the equation by $a_t$, Aubin's compactness theorem for the prescribed critical problem applies to the sequence $v_q$. Condition \eqref{eq:weighted-condition} is precisely the sharp level below which concentration at a maximum point of $f$ is impossible; see \cite[Chapter 6]{Aubin1998}. Hence a subsequence converges to a positive critical minimizer $v\in\mathcal A_f$ satisfying
\[
L_{r,t}v=\lambda_f^{r,t} f v^{N-1}.
\]
Since \eqref{eq:lambda-f-lower} gives $\lambda_f^{r,t}>0$, the rescaling
\[
u=(\lambda_f^{r,t})^{1/(N-2)}v
\]
removes the Lagrange multiplier and solves \eqref{eq:prescribed-PDE}.
\end{proof}
\begin{corollary}[Global and local prescribed-curvature tests]\label{cor:weighted-tests}
Assume \(t>0\), \(\lambda^{r,t}([\omega])>0\), and \(f\in C^\infty(M)\) with \(F:=\max_M f>0\). Either of the following conditions is sufficient for \eqref{eq:prescribed-PDE} to have a smooth positive solution.
\begin{enumerate}
\item $\int_M f\frac{\omega^n}{n!}>0$ and
\begin{equation}\label{eq:weighted-constant-test}
 \frac{\int_M\mu^{r,t}\frac{\omega^n}{n!}}{(\int_M f\frac{\omega^n}{n!})^{2/N}}
       <\frac{t\Lambda_n}{F^{2/N}};
\end{equation}
\item $n\geq3$, and at some point $P$ with $f(P)=F$,
\begin{equation}\label{eq:weighted-local}
 E_{r,t}(P)+t(n-2)\frac{\Delta f(P)}{F}<0.
\end{equation}
\end{enumerate}
\end{corollary}
\begin{proof}
The first assertion uses the constant test function, scaled to satisfy the constraint. For the second assertion,
set $m=2n$ and $H=\mu^{r,t}/a_t$. The weighted Aubin expansion for
\[
\frac{\int(|\nabla v|^2+Hv^2)}{(\int f|v|^{2m/(m-2)})^{(m-2)/m}}
\]
at a positive maximum point $P$ of $f$ gives strict inequality below the weighted spherical level if
\begin{equation}\label{eq:generic-weighted-Aubin}
4\frac{m-1}{m-2}H(P)
+\frac{m-4}{2}\frac{\Delta f(P)}{f(P)}
<\scal(g)(P);
\end{equation}
see \cite[Chapter 6]{Aubin1998}. The computation is local, so $f$ need only be positive near $P$.

Since $m=2n$ and $a_t=2t/(n-1)$, the first term in \eqref{eq:generic-weighted-Aubin} is $(2n-1)\mu^{r,t}(P)/t$, while $(m-4)/2=n-2$. Multiplication by $t$ gives exactly \eqref{eq:weighted-local}. The resulting strict inequality is \eqref{eq:weighted-condition}, and Theorem \ref{thm:weighted-threshold} finishes the proof.
\end{proof}

\subsection{A sufficient interval in the connection parameter}
\begin{corollary}\label{cor:prescribed-window}
Suppose $n\geq3$, $\Tmass([\omega])>0$, and $\lambda_{1,t}>0$. Let $r_0$ be the exact positivity endpoint from Proposition \ref{prop:positivity-endpoint}. Suppose $f(P)=F=\max f>0$ and $W(P)>0$, and set
\begin{equation}\label{eq:rA}
 r_A=1+\frac{E_{1,t}(P)+t(n-2)\Delta f(P)/F}{\ty c_n W(P)}.
\end{equation}
Then $f$ is realized as $\mu^{r,t}$ by a conformal metric whenever
\begin{equation}\label{eq:window}
 1\leq r<r_0,\qquad r>r_A.
\end{equation}
This sufficient interval is nonempty exactly when $r_A<r_0$. A convenient sufficient condition for nonemptiness is
\begin{equation}\label{eq:window-mass}
 E_{1,t}(P)+t(n-2)\frac{\Delta f(P)}F
       <\frac{\ty\lambda_{1,t}W(P)}{\Tmass([\omega])}.
\end{equation}
\end{corollary}
\begin{proof}
Suppose \(n\geq3\), \(t>0\), \(\mathcal T([\omega])>0\), and
\(\lambda_{1,t}>0\). Define
\[
r_0=1+\frac1{c_n}
\inf_{\mathcal A(u)>0}\frac{Q^{1,t}(u)}{\mathcal A(u)}.
\]
Let \(f\in C^\infty(M)\), and suppose \(f(P)=F=\max_M f>0\) and \(W(P)>0\). For $r<r_0$, the quadratic form is coercive. By \eqref{eq:E-reduction}, the left side of \eqref{eq:weighted-local} equals $\ty c_nW(P)(r_A-r)$. Thus Corollary~\ref{cor:weighted-tests} applies under \eqref{eq:window}. Since $r_0>1$, nonemptiness is equivalent to $r_A<r_0$. Finally, \eqref{eq:window-mass} gives
\[
 r_A<1+\frac{\lambda_{1,t}}{c_n\Tmass([\omega])}\leq r_0.
\]
\end{proof}

Within the positive range, $r\mapsto\lambda^{r,t}_f$ is nonincreasing, by \eqref{eq:Q-affine} and the unchanged constraint. Hence a strict weighted inequality, once obtained, persists at larger $r$ until the positivity endpoint. These statements do not imply that every prescribed function is unsolvable beyond the sufficient interval.

\subsection{Prescribed curvature on the Iwasawa threefold}
On the three-dimensional Iwasawa manifold, write
\[
 c=2(t-4r-2),\qquad q(z)=\cos(2\pi\operatorname{Re}z_1).
\]
Then $a_t=t$, $\mu^{r,t}=c$, and $\Delta q=4\pi^2q$. For $\abs\varepsilon<1$, the metric $\widehat\omega=(1+\varepsilon q)\omega_X$ has curvature
\begin{equation}\label{eq:iwasawa-prescribed}
 \mu^{r,t}(\widehat\omega)
 =\frac{c+\varepsilon(c+4\pi^2t)q}{(1+\varepsilon q)^2}.
\end{equation}
If $c>0$ and $0<\abs\varepsilon<c/(c+4\pi^2t)$, this is positive and nonconstant. The numerator is positive under the stated bound, and a nonzero quadratic denominator cannot be proportional to the linear numerator.

This example also shows why the sufficient local interval need not exist on every geometry. Here $\lambda_{r,t}>0$ exactly when $r<(t-2)/4$, while $E_{r,t}=4(3t-10r-5)$ is negative only when $r>(3t-5)/10$. The latter lower bound exceeds $(t-2)/4$ by $t/20$. At a positive maximum of $f$, the additional term $t\Delta f/f$ is nonnegative. Thus the local test gives no coercive interval on this background, although \eqref{eq:iwasawa-prescribed} produces many positive prescribed curvatures directly.

\bibliographystyle{amsplain}
\bibliography{deformed_gauduchon_yamabe_problem}
\end{document}